\documentclass[12pt]{amsart}
\newtheorem{theorem}{Theorem}[section]
\newtheorem{lemma}[theorem]{Lemma}
\newtheorem{cor}[theorem]{Corollary}
\newtheorem{prop}[theorem]{Proposition}
\theoremstyle{definition}

\newtheorem{remark}[theorem]{Remark}

\usepackage{multicol}
\usepackage{graphicx}
\usepackage{enumitem}
\usepackage{amsmath,amssymb,amsthm}
\usepackage{mathtools}

\usepackage{tikz}
\usepackage{tikz-cd}
\newcommand{\arrowIn}{\tikz \draw[-stealth, line width=0.6mm, scale=2] (-1pt,0) -- (1pt,0);}
\usetikzlibrary{arrows.meta} 

\usepackage{mathtools}
\usepackage[all]{xy}
\usepackage[pdftex,%
bookmarks=true,
breaklinks=true,
bookmarksnumbered = true,
colorlinks= true,
urlcolor= green,
anchorcolor = yellow,
citecolor=blue,
]{hyperref}   

\newtheorem{introtheorem}{Theorem}

\numberwithin{equation}{section}

\newcommand{\im}{\ensuremath{\textnormal{im}}}
\newcommand{\id}{\ensuremath{\textnormal{id}}}

\begin{document}

\title[The $R_\infty$-nilpotency degree of some lollipop GBS groups]{The $R_\infty$-nilpotency degree of some lollipop GBS groups}
\author{Wagner Sgobbi}
\address{UFES -- Universidade Federal do Espírito Santo -- Campus Goiabeiras}
\email{wagner.sgobbi@ufes.br}
\author{Dalton Couto Silva}
\address{IFSP -- Instituto Federal de São Paulo -- campus Caraguatatuba}
\email{dalton.couto@ifsp.edu.br}
\author{Daniel Vendrúscolo}
\address{UFSCar -- Universidade Federal de São Carlos}
\email{daniel.vendruscolo@ufscar.br}

\date{\today}


\begin{abstract}
We say that a group $G$ has the $R_{\infty}$ property if the number $R(\varphi)$ of twisted conjugacy classes is infinite for any automorphism $\varphi$ of $G$. The $R_{\infty}$-nilpotency degree is the least integer $c$ such that $G/\gamma_{c+1}(G)$ has property $R_{\infty}$; in case this integer does not exist, we say that $G$ has infinite $R_\infty$-nilpotency degree. In this paper, we determine the $R_{\infty}$-nilpotency degree of some lollipop GBS groups by using different approaches, such as Sylow subgroups, normal forms and presentation of nilpotent quotients, and diagrams of exact sequences. We also compute the $R_\infty$-nilpotency degree of all 2-generated lollipop GBS groups and of all GBS groups whose associated graph is a tree.
\end{abstract}

\maketitle



\section{Introduction}

Let $\varphi$ be an endomorphism of a group $G$. Then $\varphi$ gives rise to a ``twisted conjugacy'' equivalence relation on $G$ given by $x \sim_\varphi y \ \iff\ \exists\ z \in G:\ zx\varphi(z)^{-1}=y$. The number of equivalence classes is denoted by $R(\varphi)$, and we say $G$ has property $R_\infty$ (or simply $G$ has $R_\infty$) if $R(\varphi)$ is infinite for every \emph{automorphism} $\varphi$ of $G$. Twisted conjugacy began with the work of K. Reidemeister \cite{Reidemeister} and has many connections with other areas of mathematics, including fixed point theory \cite{Jiang}. We refer to the papers \cite{DaciDekimpe2} and \cite{FT}, where one can find good discussions about many aspects of the $R_\infty$ property, along with families of groups that were shown to possess it. Here, we would just like to emphasize that this topic is still very active in algebraic topology and group theory. For instance, there are recent proofs of the $R_\infty$ property for the pure Artin braid groups $P_n$ with $n \geq 3$ \cite{MS, DGO}, virtual braid groups \cite{DGO2}, Artin groups \cite{MS, SV} and braid groups of (possibly perforated) orientable surfaces \cite{DGO3}. For non-orientable surfaces, let $P_n(\mathbb{K})$, $n \geq 2$, be the pure braid group of the Klein bottle. In \cite{PS}, it is shown the existence of a normal finite index subgroup $H \lhd Aut(P_n(\mathbb K))$ with $| Aut(P_n(\mathbb K)):H| \leq \left(2 \binom{n}{2}\right)!$ such that $R(\varphi)=\infty$ for every $\varphi \in H$.

In this work, the commutator between two elements $g,h$ of a group $G$ is defined as $[g,h]=ghg^{-1}h^{-1}$ and, recursively, $[g_1,g_2,\ldots,g_n]=[[g_1,\ldots,g_{n-1}],g_n]$. If $A,B \subset G$ are non-empty, the commutator subgroup $[A,B] \leq G$ is the subgroup generated by all commutators $[a,b]$, with $a \in A, b \in B$. Now, let $\gamma_k(G)$ denote the terms of the lower central series of $G$; that is, $\gamma_1(G)=G$ and $\gamma_{k+1}(G)=[\gamma_k(G),G]$, for any $k\geq 1$. We say $G$ is nilpotent if $\gamma_k(G)=1$ for some $k \geq 1$. Twisted conjugacy has been studied for nilpotent groups in the papers \cite{DG, DaciDekimpe2, DaciDekimpe, DaciWong, Nasy, Rom}, most of which deal with nilpotent quotients. Given a group $G$, the groups $G/\gamma_k(G)$, for $k \geq 2$, are clearly nilpotent quotients of $G$. Following \cite{DaciDekimpe}, we define the $R_\infty$-nilpotency degree of $G$ as the least integer $c$ such that $G/\gamma_{c+1}(G)$ has property $R_\infty$. In case none of such quotients have property $R_\infty$, we say the $R_\infty$-nilpotency degree of $G$ is infinite. This degree was first defined in \cite{DaciDekimpe2} and relates to the following problem: it is well-known that if $G$ has a characteristic subgroup $H$ (for example, $H=\gamma_n(G)$, $n \geq 1$) such that $G/H$ has $R_\infty$, then $G$ has $R_\infty$ (Corollary \ref{characteristic}). One may ask then whether the converse is true: given an $R_\infty$ group $G$, is $R_\infty$ still valid for any of its characteristic quotients, in particular, for nilpotent ones? Notice that, since $\gamma_n(G)/\gamma_{n+1}(G)$ is characteristic in $G/\gamma_{n+1}(G)$, property $R_\infty$ for $G/\gamma_n(G)$ implies $R_\infty$ for $G/\gamma_{n+1}(G)$ (Corollary~\ref{characteristic}). Thus, computing the $R_\infty$-nilpotency degree of a given $R_\infty$ group is equivalent to knowing exactly which nilpotent quotients $G/\gamma_n(G)$ have $R_\infty$. Another important observation is that, if $G$ is finitely generated, then its $R_\infty$-nilpotency degree is at least $2$ (Corollary \ref{atleast2}). This applies to all the groups considered in our paper.

This degree has been computed in the last few years for some classes of $R_\infty$ groups, especially by K. Dekimpe and D. L. Gon\c calves. In \cite{DG}, nonabelian finitely generated free groups $F_n$ are studied and it is shown that their $R_{\infty}$-nilpotency degree is $2n$. In fact, this result is what motivated the definition of $R_\infty$-nilpotency degree in \cite{DaciDekimpe2}. There, the same authors study hyperbolic surface groups (which are $R_\infty$ groups, see \cite{Felsh, LevittLustig}) and show that the $R_{\infty}$-nilpotency degree of $\pi_1(S_g)$ is 4, for $S_g$ an orientable surface of genus $g\geq 2$, and that the $R_{\infty}$-nilpotency degree of $\pi_1(N_g)$ is $2(g-1)$, for $N_g$ a connected sum of $g\geq 3$ projective planes. A consequence of their work for Nielsen fixed point theory, for instance, is the discovery of certain nilmanifolds on which every self-homotopy equivalence can be deformed into a fixed point free map.

The specific motivation for our work comes, however, from the non-hyperbolic family of the Baumslag--Solitar groups
\[
BS(m,n)=\left<a,t\ |\ ta^mt^{-1}=a^n \right>,
\]
for $m,n$ non-zero integers, which can be assumed to satisfy $0< m \leq |n|$ (see \cite{DaciFel}). According to \cite{Levitt}, these groups were introduced by G. Baumslag and D. Solitar \cite{BS} as examples of non-Hopfian groups, that is, groups $G$ with non-injective epimorphisms from $G$ to itself. More than 60 years later, this is still a very important family in combinatorial and geometric group theory. In \cite{DaciDekimpe}, the $R_\infty$-nilpotency degree of $BS(m,n)$ is computed.

\begin{introtheorem}[Theorem 5.4 in \cite{DaciDekimpe}] \label{DaciDekimpe} Let $0 < m \leq |n|$ and take $d=\gcd(m,n)$. Let $p$ denote the largest integer such that $2^p \mid 2\frac{m}{d} + 2$. Then, the $R_\infty$-nilpotency degree $r$ of $BS(m,n)$ is given by

\begin{itemize}
    \item If $n < 0$ and $n \neq -m$, then $r = 2$.
\item If $n = -m$ then $r = \infty$.
\item If $n = m$ then $r = \infty$.
\item If $n - m = d$, then $r = \infty$.
\item If $n - m = 2d$, then $2 \leq r \leq p + 2$.
\item If $n - m \geq 3d$, then $r = 2$.
\end{itemize}
\end{introtheorem}

The authors of \cite{DaciDekimpe} comment that, after their results on $BS(m,n)$, it would be a natural task to investigate the $R_\infty$-nilpotency degree for the family of (non-elementary) Generalized Baumslag--Solitar groups (or GBS groups). It is relevant to say that, in a recent work \cite{SSV}, we added one more important family of generalizations of Baumslag--Solitar groups to this discussion, namely, the Generalized Solvable Baumslag--Solitar groups $\Gamma_n$. Let $n\geq 2$ be an integer with prime decomposition $n={p_1}^{y_1}\cdots{p_r}^{y_r}$, where the $p_i$ are pairwise distinct and $y_i>0$. We consider the Generalized Solvable Baumslag--Solitar group
\[
\Gamma_n=\left<a,t_1,\ldots,t_r\ |\ t_it_j=t_jt_i,\ i\neq j,\ t_ia{t_i}^{-1}=a^{{p_i}^{y_i}},\ i=1,\ldots,r \right>.
\] In particular, if $n=p^y$ for a single prime $p$, then $\Gamma_n=BS(1,n)$.

\begin{introtheorem}[\cite{SSV}, Theorem 4.5]\label{teoprincipal}
Let an integer $n \geq 1$ have a prime decomposition with at least two prime numbers involved. Then the $R_\infty$-nilpotency degree of the Generalized Solvable Baumslag--Solitar group $\Gamma_n$ is infinite.
\end{introtheorem}

Therefore, having succeeded in showing that the $R_{\infty}$-nilpotency degree of $\Gamma_n$ is infinite, we began focusing on GBS groups. Recall that a GBS group is a finitely generated group $G$ which acts (with no edge inversions) on a tree with all edge and vertex stabilizers being infinite cyclic. Equivalently \cite{Serre}, a GBS group is the fundamental group $G$ of a finite graph of groups $\Gamma$ whose vertex and edge groups are all infinite cyclic (what we call a ``graph of $\mathbb{Z}$'s''). We describe this terminology as outlined in \cite{Levitt}: choose a finite connected graph $\Gamma$. A pair $\epsilon=(e,\overline{e})$ of opposite edges is called a non-oriented edge of $\Gamma$. We associate to each vertex $v$ (resp. to each non-oriented edge $\epsilon$) an infinite cyclic group $G_v\simeq \mathbb{Z}$ (resp. $G_\epsilon\simeq \mathbb{Z}$). Choose respective generators $x_v$ for all vertex groups $G_v$, as well as injective homomorphisms from the edge groups to the vertex groups of the vertices of the edge. These homomorphisms are multiplications by non-zero integers. An oriented edge $e$ thus has a label $\lambda_e \in
\mathbb{Z} - \{0\}$, describing the inclusion of the edge group $G_e$ into the vertex group $G_{o(e)}$
at the origin of $e$. We visualize $\Gamma$ as a labeled graph, with the label $\lambda_e$ pictured near the origin $o(e)$. A non-oriented edge carries two labels, one near either endpoint $o(e), o(\overline{e})$ of $\epsilon$ and we say that $\epsilon$ (or $e$) is a $(\lambda_e,\lambda_{\overline{e}})$-edge. We say an edge is a loop if its endpoints are equal, a segment if they are distinct.

Given such a graph of $\mathbb{Z}$'s, one can describe a finite presentation of the associated GBS group in the following way: choose a maximal subtree $T \subset \Gamma$ and fix one generator $t_\epsilon$ for each non-oriented edge $\epsilon$ not in $T$. The generators of $G$ are the $x_v$ for every vertex $v$ (called \emph{vertex letters}) and the $t_\epsilon$ above (called \emph{stable letters}). Each non-oriented edge $\epsilon=(e,\overline{e})$ of $\Gamma$ contributes one relation. If $\epsilon$ is contained in $T$, the relation is $x_{o(e)}^{\lambda_e} = x_{o(\overline{e})}^{\lambda_{\overline{e}}}$. If $\epsilon$ is not in $T$, the relation is $t_ex_{o(e)}^{\lambda_e}t_e^{-1} = x_{o(\overline{e})}^{\lambda_{\overline{e}}}$. GBS groups are known to be torsion-free \cite{Levitt}.

In the case where the graph $\Gamma$ consists of only one vertex and one $(m,n)$-edge attached to it, the group $G=BS(m,n)$ is a Baumslag--Solitar group (hence the name GBS groups). In particular, the groups $BS(1,1) \simeq \mathbb{Z}^2$ and the fundamental group of the Klein bottle $\pi_1(\mathbb{K})\simeq BS(1,-1)$ are called \emph{elementary} GBS groups, along with the GBS group $G=\mathbb{Z}$ (the case where $\Gamma$ is a single point). All other GBS groups are called non-elementary, and it is well-known
that these groups possess property $R_\infty$ \cite[Proposition 2.7]{Levitt}, so our aim is to investigate their $R_\infty$-nilpotency degree. We begin with the case where the graph $\Gamma$ has topological rank zero, that is, $\Gamma$ is a tree, and show that its $R_{\infty}$-nilpotency degree is infinite (Theorem \ref{rank0}). Since the groups $BS(m,n)$ (the simplest cases of topological rank 1) were studied in \cite{DaciDekimpe}, we focused on the lollipop GBS groups, which are the simplest cases of GBS groups with topological rank 1 that are not Baumslag--Solitar groups. The lollipop GBS group $G$ will be denoted by $\textnormal{loll}(a,b,c,d)$ and is the GBS group associated with the following graph.

\begin{center}
    \begin{figure}[ht]
\hfill
\begin{tikzpicture}[scale=1.5, very thick]

\foreach \j in {0,2}
{\filldraw[black] (\j,0) circle (2pt);};

\draw[thick] (0,0)--(2,0);
\begin{scope}[shift={(3,0)}]
    \def\r{1} 
    \def\k{0.55228475} 

    \draw[thick] (\r,0) .. controls (\r,\k*\r) and (\k*\r,\r) .. (0,\r) 
                .. controls (-\k*\r,\r) and (-\r,\k*\r) .. (-\r,0) 
                .. controls (-\r,-\k*\r) and (-\k*\r,-\r) .. (0,-\r) 
                .. controls (\k*\r,-\r) and (\r,-\k*\r) .. (\r,0)node[sloped, pos=0.99, allow upside down]{\arrowIn};;; 
\end{scope}

\node at (0.25,-0.2) {\large $x$};
\node at (0.25,0.2) {\large \textcolor{blue}{$a$}};
\node at (1.75,0.2) {\large \textcolor{blue}{$b$}};
\node at (1.75,-0.2) {\large $y$};
\node at (2.25,0.2) {\large \textcolor{blue}{$c$}};
\node at (2.25,-0.2) {\large \textcolor{blue}{$d$}};
\node at (4.25,0) {\large $t$};

\end{tikzpicture}
\hspace*{\fill}
\end{figure}
\end{center}
Here, $a,b,c,d$ are non-zero integers and $x,y,t$ denote the group generators. In this case, the group $G$ has a presentation
\[
G=\textnormal{loll}(a,b,c,d)=\langle x,y,t \: | \: x^a=y^b, \: ty^ct^{-1}=y^d\rangle.
\]

In terms of group splittings \cite{Serre}, let $A=\langle x^a\rangle=\langle y^b \rangle \leq G$. Then $G$ can be seen either as an HNN extension of an amalgamation $\mathbb{Z}*_A\mathbb{Z}$, or as an amalgamation of the form $\mathbb{Z}*_A BS(c,d)$. Because of the isomorphisms $BS(d,c) \simeq BS(c,d) \simeq BS(-c,-d)$, we can assume without loss of generality that $0< c\leq |d|$ (see \cite{DaciFel}). Since $\textnormal{loll}(a,b,c,d) \simeq \textnormal{loll}(-a,-b,c,d)$, we may also suppose $a \geq 1$. Furthermore, we assume $a>1$ throughout the paper, otherwise $G$ would be (isomorphic to) a Baumslag--Solitar group.


This paper is organized as follows. In Section \ref{preliminaries}, we recall some facts on Reidemeister numbers, lower central series and nilpotent quotients. In Section \ref{firstcases}, we use suitable diagrams of exact sequences to compute the $R_\infty$-nilpotency degree of GBS groups associated with trees and of lollipop GBS groups $\textnormal{loll}(a,b,c,d)$, where either $c=|d|$ or $|c-d|=\gcd(c,d)$. All these groups are shown to have infinite $R_\infty$-nilpotency degree, as one can see below.

\begin{introtheorem}[Theorem \ref{rank0}]
Let $G$ be a GBS group with the associated graph $\Gamma$. If $\Gamma$ is a tree, then the $R_{\infty}$-nilpotency degree of $G$ is infinite.
\end{introtheorem}

\begin{introtheorem}[Theorem \ref{casoc=|d|}]\label{casoc=|d|intro}
If $c=|d|$, then the $R_{\infty}$-nilpotency degree of $G=\textnormal{loll}(a,b,c,d)$ is infinite.
\end{introtheorem}

\begin{introtheorem}[Theorem \ref{c-d=mdc}]\label{c-d=mdcintro}
If $|c-d|=\gcd(c,d)$, then the $R_{\infty}$-nilpotency degree of $G=\textnormal{loll}(a,b,c,d)$ is infinite.
\end{introtheorem}

In Section \ref{normalformsection}, we obtain a presentation for $G/\gamma_3(G)$ by working with the free nilpotent quotient $F/\gamma_3(F)$, and then use normal forms for polycyclic groups \cite{Sims} to derive criteria under which the $R_{\infty}$-nilpotency degree of $\textnormal{loll}(a,b,c,d)$ is $2$.

\begin{introtheorem}[Theorem \ref{pirulitosindice2}]\label{|c-d|oddintro}
If $G=\textnormal{loll}(a,b,c,d)$ satisfies $\gcd(a,c-d)=1=\gcd(c,d)$, and the number $1<|c-d|$ is odd, then the $R_{\infty}$-nilpotency degree of $G$ is $2$.
\end{introtheorem}

\begin{introtheorem}[Theorem \ref{notdivideb}]\label{notdividebintro}
If $G=\textnormal{loll}(a,b,c,d)$ satisfies $\gcd(a,c-d)=1=\gcd(a,b)$, $|c-d|>\gcd(c,d)$ and $c-d$ does not divide $b$, then the $R_{\infty}$-nilpotency degree of $G$ is $2$.
\end{introtheorem}

We then use Theorem \ref{notdivideb} to refine Theorem \ref{DaciDekimpe} by 12,5\% (or $\frac{1}{6} \cdot \frac{3}{4}$ of it).

\begin{introtheorem}[Theorem \ref{refinement}, a 12,5\% refinement of Theorem \ref{DaciDekimpe}]
Let $0<m\leq |n|$ with $m\neq n$ and suppose $d=\gcd(m,n)>1$. If $n-m=2d$ and $m,n$ are not both even, then the $R_{\infty}$-nilpotency degree of $BS(m,n)$ is $2$.
\end{introtheorem}

In Section \ref{Sylow}, we use the normal forms from Section \ref{normalformsection} and basic facts about Sylow subgroups \cite{Hungerford}. We show that the $R_{\infty}$-nilpotency degree of $\textnormal{loll}(a,b,c,d)$ is $2$ for some new cases under mild divisibility conditions. Recall that a \emph{Hall divisor} of an integer $n$ is defined as a divisor $m$ of $n$ such that $m$ and $n/m$ are coprime. For an integer $k\geq 1$, we denote by $P(k)$ the (finite) set of prime numbers involved in the prime decomposition of $k$.

\begin{introtheorem}[Theorem \ref{Sylow1}]\label{sylowintro}
    Let $G=\textnormal{loll}(a,b,c,d)$. Suppose that $n=\gcd(a,b,c-d)=1$, that $m=\gcd(c,d)$ is a Hall divisor of $|c-d|$ and that $P(a)\cap P(|c-d|) \subset P(m)$. Then, the $R_{\infty}$-nilpotency degree of $G$ is $2$ in any of the following cases:
\begin{itemize}
    \item $d<0$ and $d \neq -c$;
    \item $d>0$ and $d-c \geq 3m$.
\end{itemize}
\end{introtheorem}

We note that Theorem \ref{sylowintro} subsumes Theorem \ref{|c-d|oddintro}. Nevertheless, we state Theorem \ref{|c-d|oddintro} separately because its hypotheses are formulated directly in terms of the elementary arithmetic of $c$ and $d$ and are easier to check; also, it fits naturally alongside the neighboring criteria,
such as Theorem \ref{c-d=mdcintro} and Theorem \ref{notdividebintro}.

Finally, Section \ref{2generated} describes the $R_\infty$-nilpotency degree of 2-generated lollipop GBS groups.

\begin{introtheorem}[Theorem \ref{2geradosfinal}]
Let $G=\textnormal{loll}(a,b,c,d)$ be a 2-generated lollipop GBS group. Then, $G$ and $BS(ac,ad)$ have the same $R_{\infty}$-nilpotency degree $r$. Precisely, let $p$ denote the largest integer such that $2^p$ divides $2\frac{c}{\gcd(c,d)}+2$. We have:

\begin{itemize}
    \item If $d<0$ and $d \neq -c$, then $r=2$.
    \item If $d=c$ or $d=-c$, then $r=\infty$.
    \item If $d-c=\gcd(c,d)$, then $r=\infty$.
    \item If $d-c=2\gcd(c,d)$ and $ac,ad$ are not both even, then $r=2$.
    \item If $d-c=2\gcd(c,d)$ and $ac,ad$ are both even, then $2\leq r\leq p+2$.
    \item If $d-c\geq 3\gcd(c,d)$, then $r=2$.
\end{itemize}
\end{introtheorem}

Note that, throughout the paper, many different approaches were used. This is another reason why we believe this research can be considered a first step toward computing the $R_\infty$-nilpotency degree for broader classes of non-elementary GBS groups, beyond those considered in \cite{DaciDekimpe}.

In this work, the center of a group $G$ will be denoted by $Z(G)$, and the torsion subgroup of a nilpotent group $G$ will be denoted by $\tau(G)=\{g \in G\ |\ g^n=1\ \text{for some}\ n \geq 1\}$. If the context is clear, we may write only $\gamma_k$ instead of $\gamma_k(G)$.

\subsection*{Acknowledgements} We thank Professor Agatha Sacramento Rodrigues (Statistics Department -- UFES -- Brazil) for the assistance with the R software \cite{R} (see Remark \ref{howmuch}). We thank Professor Carolina de Miranda e Pereiro (Mathematics Department -- UFES -- Brazil) for the assistance with the TikZ package. We thank the referee for the valuable suggestions about the references and the text format. This study was financed, in part, by the São Paulo Research Foundation (FAPESP), Brazil. Process Number 2016/24707-4.

\section{Preliminaries}\label{preliminaries}

We now recall properties of twisted conjugacy and of nilpotent quotients of a group $G$. We begin with some useful tools for computing Reidemeister numbers. Although they are well-known, we provide proofs for completeness. For the next lemma, we use abelian notation.

\begin{lemma}\label{reidemeisterabelian}
Let $\varphi$ be an endomorphism of an abelian group $G$. Then
\begin{itemize}
    \item[1)] $R(\varphi)$ equals the cardinality of the quotient group $G/\im(\varphi-\id)$.
    \item[2)] In particular, if $G$ is finitely generated abelian and $\varphi=- \id$, then $R(\varphi)$ is finite.
\end{itemize}
\end{lemma}

\begin{proof}
In fact, given $g,h \in G$, we have
\begin{eqnarray*}
g \sim_\varphi h 	&\Longleftrightarrow & \exists z \in G: z + g - \varphi(z)= h\\
					&\Longleftrightarrow & \exists z \in G:g - h=(\varphi-\id)(z)\\
					&\Longleftrightarrow & g - h \in \im(\varphi-\id).		
\end{eqnarray*} So, $R(\varphi)$ is clearly the cardinality of $G/\im(\varphi-\id)$ and item 1) is proved. Suppose now that $G$ is finitely generated and $\varphi=-\id$. By item 1), $R(\varphi)$ equals the cardinality of $G/2G$, which is clearly a finitely generated torsion abelian group and hence finite. This shows item 2) and completes the proof.
\end{proof}

\begin{cor}\label{atleast2}
The $R_\infty$-nilpotency degree of any finitely generated group is at least 2.
\end{cor}

\begin{proof}
If $G$ is finitely generated, then $G/\gamma_2(G)$ is a finitely generated abelian group. Since $\varphi=-\id$ is clearly an automorphism of any abelian group, it follows immediatly from item 2) of Lemma \ref{reidemeisterabelian} that $G/\gamma_2(G)$ does not have property $R_\infty$. Hence, the $R_\infty$-nilpotency degree of $G$ is at least 2.
\end{proof}

\begin{lemma}\label{cavelemma}
Consider the following commutative diagram

\[
\xymatrix{
  1 \ar[r] & A \ar[d]_-{\varphi'} \ar[r]^-{i} & B \ar[d]_-{\varphi} \ar[r]^-{p} & C \ar[d]^-{\overline{\varphi}} \ar[r] & 1 \\
  1 \ar[r] & A \ar[r]_-{i} & B \ar[r]_-{p} & C \ar[r] & 1
}
\]
where the rows are short exact sequences of groups and the vertical arrows are group endomorphisms. Then
\begin{itemize}
\item[1)] If $R(\overline{\varphi})=\infty$, then $R(\varphi)=\infty$;
\item[2)] If (the image of) $A$ is central in $B$, then $R(\varphi)\leq R(\varphi')R(\overline{\varphi})$. In particular, if both $R(\varphi')$ and $R(\overline{\varphi})$ are finite, then $R(\varphi)$ is finite.
\end{itemize}
\end{lemma}

\begin{proof}
Denote by $[b]_\varphi$ and $[c]_{\overline{\varphi}}$ the twisted conjugacy classes of any $b \in B$ and $c \in C$, respectively. Let $\mathcal{R}(\varphi)$ and $\mathcal{R}(\overline{\varphi})$ denote the sets of twisted conjugacy classes of $\varphi$ and $\overline{\varphi}$, respectively. There is a well-defined map $\hat{p}\colon \mathcal{R}(\varphi) \to \mathcal{R}(\overline{\varphi})$ with $\hat{p}([b]_\varphi)=[p(b)]_{\overline{\varphi}}$. Indeed, for $b,b' \in B$, we have \begin{eqnarray*}
[b]_\varphi=[b']_\varphi 	&\Rightarrow & \exists z \in B:zb\varphi(z)^{-1}=b'\\
&\Rightarrow & p(z)p(b)p(\varphi(z))^{-1}=p(b')\\
&\Rightarrow & p(z)p(b)\overline{\varphi}(p(z))^{-1}=p(b')\\
&\Rightarrow & [p(b)]_{\overline{\varphi}}=[p(b')]_{\overline{\varphi}}.
\end{eqnarray*} Since $p$ is surjective, it is clear that $\hat{p}$ is also surjective, which implies $R(\varphi) \geq R(\overline{\varphi})$. Therefore, $R(\overline{\varphi})=\infty$ implies $R(\varphi)=\infty$ and item 1) is proved. For the proof of item 2), see \cite[Lemma 2.1]{Senden25}.
\end{proof}

\begin{cor}\label{characteristic}
If $H$ is characteristic in $G$ and $G/H$ has $R_\infty$, then $G$ has $R_\infty$. In particular, if a group $G$ is such that $G/\gamma_n(G)$ has $R_\infty$ for some $n \geq 2$, then $G/\gamma_m(G)$ has $R_\infty$ for any $m\geq n$.
\end{cor}

\begin{proof}
If $H$ is characteristic in $G$ and $\varphi$ is an automorphism of $G$, there is a well-defined commutative diagram as the one in Lemma \ref{cavelemma}, with $A=H$, $B=G$, $C=G/H$ and $\varphi'$, $\overline{\varphi}$ being the naturally induced automorphisms on $A$ and $C$, respectively. Since $G/H$ has $R_\infty$, we have $R(\overline{\varphi})=\infty$. Hence, $R(\varphi)=\infty$ by item 1) of Lemma \ref{cavelemma}, which shows the first claim. For the last claim, recall that  $\gamma_{n}(G)/\gamma_{n+1}(G)$ is characteristic in $G/\gamma_{n+1}(G)$, with quotient
\[
\dfrac{G/\gamma_{n+1}(G)}{\gamma_{n}(G)/\gamma_{n+1}(G)}\simeq \dfrac{G}{\gamma_{n}(G)}.
\]By what we have just shown, if $G/\gamma_{n}(G)$ has $R_\infty$, then $G/\gamma_{n+1}(G)$ has $R_\infty$ as well. By performing an easy induction, one can show that $G/\gamma_m(G)$ has $R_\infty$ for any $m\geq n$.
\end{proof}

The next two results are shown by easy adaptations of \cite[Theorems 5.3 and 5.4]{MagKarSol} to our definition of commutators, that is, $[g,h]=ghg^{-1}h^{-1}$.

\begin{prop}\label{commut3}
Let $k,m,n \geq 1$ be integers and let $x,y,z \in G$ be elements of a group $G$ such that $x \in \gamma_k(G)$, $y \in \gamma_m(G)$ and $z \in \gamma_n(G)$. Then
\begin{itemize}
\item[a)] $xy=yx\  \pmod{ \gamma_{k+m}(G)}$;
\item[b)] $[x,yz]=[x,y][x,z]\  \pmod{\gamma_{k+m+n}(G)}$;
\item[c)] $[xy,z]=[x,z][y,z]\pmod{\gamma_{k+m+n}(G)}$.
\end{itemize}
\end{prop}

\begin{prop}\label{commutmks}
If $G$ is generated by elements $x_1,\ldots,x_r$, then, for any $k \geq 1$, the group $\gamma_k(G)/\gamma_{k+1}(G)$ is abelian and finitely generated by the cosets of the $k$-fold commutators \newline $[x_{i_1},\ldots,x_{i_k}]$, where $1 \leq i_j \leq r$, with the convention that a $1$-fold commutator is just $x_i$.
\end{prop}

\begin{prop}\label{torfinita}
Let $G$ be finitely generated. If $\gamma_n/\gamma_{n+1}$ is finite for some $n \geq 1$, then $\gamma_k/\gamma_m$ is finite for all $n \leq k \leq m$.
\end{prop}

\begin{proof}
First, we claim that $\gamma_k/\gamma_{k+1}$ is finite for any $k \geq n$. By assumption, it holds for $k=n$. Suppose it holds for some $k\geq n$. It follows from Proposition \ref{commutmks} that $\gamma_{k+1}/\gamma_{k+2}$ is generated by all the elements $[z,x]\gamma_{k+2}$, where $z\gamma_{k+1}$ is a generator of $\gamma_k/\gamma_{k+1}$ and $x\gamma_2$ a generator of $G/\gamma_2$. Since $z\gamma_{k+1}$ has finite order in $\gamma_k/\gamma_{k+1}$, there exists $l \geq 1$ such that $z^l \in \gamma_{k+1}$. By Proposition \ref{commut3}, it follows that $[z,x]^l\gamma_{k+2}=[z^l,x]\gamma_{k+2}=\gamma_{k+2}$,
so all generators of $\gamma_{k+1}/\gamma_{k+2}$ are of finite order. Since $\gamma_{k+1}/\gamma_{k+2}$ is finitely generated and abelian (Proposition \ref{commutmks}), it is a finite group, hence the claim holds for $k+1$. Therefore, the induction step is complete and the claim is true for all $k \geq n$.

Finally, given $n \leq k \leq m$, by the previous claim and Lagrange's Theorem we obtain
\[
|\gamma_k/\gamma_m|=|\gamma_k/\gamma_{k+1}||\gamma_{k+1}/\gamma_{k+2}|\cdots|\gamma_{m-1}/\gamma_m|<\infty,
\]as desired.
\end{proof}

\section{Trees and some simple lollipop GBS groups} \label{firstcases}

We first study the case of non-elementary GBS groups associated with (simplicial) trees. Recall that a tree is a non-empty and connected graph $\Gamma$ with no closed circuits. That means the GBS group $G$ has no stable letters in its presentation. For these groups, we compute the $R_{\infty}$-nilpotency degree and prove that it is infinite. Some simple examples are torus knot groups of the form $G(a,b)=\left<x,y\ |\ x^a=y^b\right>$, where $a$ and $b$ are coprime positive integers. These groups are the fundamental groups of the complement of certain torus knots in $\mathbb{R}^3$. For instance, $G(2,3)$ is the well-known trefoil knot group.




\begin{lemma}\label{powers}
Let $G$ be a GBS group represented by a labeled tree $\Gamma$. Then for any vertices $v,w \in \Gamma$, there exist non-zero integers $\alpha,\beta$ such that $x_v^\alpha=x_w^\beta$.
\end{lemma}

\begin{proof}
We proceed by induction on the geodesic distance $n\geq 0$ between $v$ and $w$ in $\Gamma$. The case $n=0$ is obvious. If $n=1$, then $v$ and $w$ are adjacent and the lemma is true by the presentation of $G$. Suppose the lemma is true for some $n \geq 1$ and let us show it for $n+1$. Let $\gamma$ be a geodesic in $\Gamma$ of length $n+1$ from $v$ to $w$. Let $u \in \gamma$ be the vertex adjacent to $w$. Then the distance between $v$ and $u$ is $n$. By induction hypothesis, we have $x_v^k=x_u^l$, for non-zero integers $k,l$. By the presentation of $G$, we have $x_u^p=x_w^q$, for non-zero integers $p,q$. 
Then, for the non-zero integers $\alpha=pk$ and $\beta=ql$ we have
\[
x_v^\alpha=x_v^{pk}=x_u^{pl}=x_w^{ql}=x_w^{\beta},
\] as desired. This completes the proof of the lemma.
\end{proof}

\begin{lemma}\label{centralsubgroup}
Let $G$ be a GBS group represented by a labeled tree $\Gamma$. Then there exists an infinite cyclic subgroup $C \leq Z(G)$ such that, for every vertex $v$ of $\Gamma$, some positive power of the vertex generator $x_v$ belongs to $C$.
\end{lemma}

\begin{proof}
Choose a vertex $v_0$ of $\Gamma$, and write $x_0 = x_{v_0}$. By Lemma \ref{powers}, for every vertex $v$ of $\Gamma$ there exist non-zero integers $\alpha_v$, $\beta_v$ such that $x_0^{\alpha_v}=x_v^{\beta_v}$.
Since $\Gamma$ is finite, choose a positive integer $L$ divisible by $|\alpha_v|$ for every vertex $v$. Then
\[
z \coloneq x_0^L = (x_0^{\alpha_v})^{L/\alpha_v} = (x_v^{\beta_v})^{L/\alpha_v}
= x_v^{\beta_v L /\alpha_v}
\]
for every vertex $v$. Set $C = \langle z\rangle$. Since the vertex group $\langle x_0 \rangle$ embeds in $G$ and is infinite cyclic, $C$ is infinite cyclic. Moreover, $z$ commutes with every vertex generator, because $z$ lies in every vertex group. Since $\Gamma$ is a tree, there are no stable letters in the presentation of $G$, thus $G$ is generated by its vertex generators. Hence $z \in Z(G)$, and therefore $C \leq Z(G)$. Moreover, 
$x_v^{\beta_v L/\alpha_v} \in C$ for any vertex generator $x_v$, by the above.
\end{proof}


\begin{theorem}\label{rank0}
Let $G$ be a GBS group represented by a labeled tree $\Gamma$. Then the $R_{\infty}$-nilpotency degree of $G$ is infinite.
\end{theorem}

\begin{proof}

Let $n \geq 2$ and write $\overline{G}= G/\gamma_n(G)$, with projection $\pi\colon G \to \overline{G}$. By Lemma \ref{centralsubgroup}, there exists an infinite cyclic subgroup $C \leq Z(G)$ such that, for every vertex $v$ of $\Gamma$, some positive power of the vertex generator $x_v$ belongs to $C$. Set $\overline{H} = \pi(C) \leq \overline{G}$. Then $\overline{H}$ is cyclic and central in $\overline{G}$. For every vertex $v$, choose $m(v) \geq 1$ such that $x_v^{m(v)} \in C$. Then the image of $\pi(x_v)$ in $\overline{G}/\overline{H}$ has finite order. Since $\Gamma$ is a tree, the vertex generators generate $G$, and hence
their images generate $\overline{G}/\overline{H}$. Thus $\overline{G}/\overline{H}$ is a finitely generated nilpotent group generated by elements of finite order, so it is finite.


Now, we consider the automorphism
$$\varphi\colon G \to G$$
such that $\varphi(w)=w^{-1}$ for every vertex letter $w$, and let $\overline{\varphi}\colon \overline{G}\to \overline{G}$ be the induced automorphism in $\overline{G}$. It is clear that $\overline{H}$ is invariant by $\overline{\varphi}$, thus we have the following commutative diagram:

\[
\xymatrix{
  1 \ar[r] & \overline{H} \ar[d]_-{\overline{\varphi}|_{\overline{H}}} \ar[r]^-{} & \overline{G} \ar[d]_-{\overline{\varphi}} \ar[r]^-{} & \overline{G}/\overline{H} \ar[d]^-{\overline{\overline{\varphi}}} \ar[r] & 1 \\
  1 \ar[r] & \overline{H} \ar[r]_-{} & \overline{G} \ar[r]_-{} & \overline{G}/\overline{H} \ar[r] & 1
}
\]
Since $\overline{G}/\overline{H}$ is finite, we conclude that $R(\overline{\overline{\varphi}})< \infty$. Furthermore, since $\overline{H}$ is cyclic and $\overline{\varphi}|_{\overline{H}}$ acts as multiplication by $-1$ on $\overline{H}$, it follows from item 2) of Lemma \ref{reidemeisterabelian} that $R(\overline{\varphi}|_{\overline{H}})< \infty$. Since $\overline{H}$ is central in $\overline{G}$, $R(\overline{\varphi})$ must be finite by Lemma \ref{cavelemma}, item 2). Thus, $\overline{G}=G/\gamma_n(G)$ does not have property $R_\infty$. Since $n$ is arbitrary, the $R_{\infty}$-nilpotency degree of $G$ is infinite.
\end{proof}

Now we deal with the lollipop GBS groups. Recall that these groups admit a presentation of the form
\[
G=\textnormal{loll}(a,b,c,d)=\langle x,y,t \: | \: x^a=y^b, \: ty^ct^{-1}=y^d\rangle,
\]where $a,b,c,d$ are non-zero integers with $0 < c\leq |d|$ and $a>1$. Following the case division indicated by K. Dekimpe and D. L. Gonçalves in Theorem \ref{DaciDekimpe}, we begin by studying the case where $c=|d|$. The lollipop GBS groups that satisfy this condition are called \emph{unimodular}. In general, a GBS group $G$ is said to be unimodular if, for any $g,h \in G$ with $h\neq 1$, the equality $gh^pg^{-1}=h^q$ implies $|p|=|q|$. We refer to \cite{Levitt} for further information on unimodular GBS groups.

\begin{theorem} \label{casoc=|d|}
If $c=|d|$, then the $R_{\infty}$-nilpotency degree of $G=\textnormal{loll}(a,b,c,d)$ is infinite.
\end{theorem}

\begin{proof}
We first claim that $\gamma_2/\gamma_3$ is finite. Since this is a finitely generated abelian group, it suffices to show that all its generators are of finite order, so we claim this is true. By Proposition \ref{commutmks}, $\gamma_2/\gamma_3$ is generated by $[x,y]\gamma_3, [x,t]\gamma_3$ and $[y,t]\gamma_3$. By the relations on $G$ and Proposition \ref{commut3}, we have $[x,y]^a\gamma_3=[x^a,y]\gamma_3=[y^b,y]\gamma_3=\gamma_3$. Similarly, we have $[x,t]^a\gamma_3=[y,t]^b\gamma_3$, so finite order for $[y,t]\gamma_3$ implies finite order for $[x,t]\gamma_3$. For $[y,t]\gamma_3$, we observe that, if $c=-d$, then $[y,t]^{d-c}\gamma_3=[y^{d-c},t]\gamma_3=[[t,y^c],t]\gamma_3=\gamma_3$ and $[y,t]\gamma_3$ has finite order. If $c=d$, then $ty^ct^{-1}=y^c$. Hence, $[t,y^c]=1$. But in $\gamma_2/\gamma_3$, we have $\gamma_3=[t,y^c]\gamma_3=[t,y]^c\gamma_3$, so $[y,t]\gamma_3$ also has finite order and the claim is true. By Proposition \ref{torfinita}, $\gamma_n/\gamma_{n+1}$ is finite for every $n \geq 2$.

Now, since $c=|d|$, it is easy to check that the map $\varphi\colon G \to G$ with $x \mapsto x^{-1}$, $y \mapsto y^{-1}$ and $t \mapsto t^{-1}$ is a well-defined automorphism of $G$. For every $n \geq 2$, let $\varphi_n$ denote the induced automorphism of $\varphi$ on the quotient $G/\gamma_n$. We will show that $R(\varphi_n)$ is finite by induction on $n\geq 3$. For $n=3$, we have the following natural commutative diagram:

\[
\xymatrix{
  1 \ar[r] & \gamma_2/\gamma_3 \ar[d]_-{\varphi_3'} \ar[r]^-{} & G/\gamma_3 \ar[d]_-{\varphi_3} \ar[r]^-{} &  G/\gamma_2 \ar[d]^-{\varphi_2} \ar[r] & 1 \\
  1 \ar[r] & \gamma_2/\gamma_3 \ar[r]_-{} & G/\gamma_3 \ar[r]_-{} & G/\gamma_2 \ar[r] & 1
}
\] Clearly, $R(\varphi_3')$ is finite since $\gamma_2/\gamma_3$ is finite. Furthermore, $G/\gamma_2$ is a finitely generated abelian group and $\varphi_2=-\id$ (in abelian notation), so $R(\varphi_2)$ is finite by Lemma \ref{reidemeisterabelian}, item 2). Since $\gamma_2/\gamma_3 \subset Z(G/\gamma_3)$, it follows that $R(\varphi_3)$ is finite by Lemma \ref{cavelemma}, item 2). The induction step follows similarly, using the commutative diagram
\[
\xymatrix{
  1 \ar[r] & \gamma_n/\gamma_{n+1} \ar[d]_-{\varphi_{n+1}'} \ar[r]^-{} & G/\gamma_{n+1} \ar[d]_-{\varphi_{n+1}} \ar[r]^-{} &  G/\gamma_n \ar[d]^-{\varphi_n} \ar[r] & 1 \\
  1 \ar[r] & \gamma_n/\gamma_{n+1} \ar[r]_-{} & G/\gamma_{n+1} \ar[r]_-{} & G/\gamma_n \ar[r] & 1,
}
\] the finiteness of $\gamma_n/\gamma_{n+1}$ and the induction hypothesis. This completes the proof.
\end{proof}

From now on, we can assume $1\leq c < |d|$. In particular, $|c-d|\geq 1$.

\begin{lemma} \label{comutatpotencomut}
    Let $G$ be a group, $n\geq 3$, $k,l\geq 1$ be integers, and let $g,h\in G$. If $g^k \gamma_n$ is central in $G/\gamma_n(G)$, then $[g^{kl},h]\gamma_{n+1}=[g^k,h^l]\gamma_{n+1}$.

    \begin{proof}
   We perform an induction on $l$. The case $l=1$ is clear. For the induction case, we will need the following group identities (see e.g. \cite{CMZ, MagKarSol}), which hold for all $x,y,z$ in a group $G$:
    \[ [xy,z]=[y,z][z,y,x][x,z] \qquad \mbox{and} \qquad [x,yz]=[x,y][x,z][z,x,y].\]
    We also need to observe that, if $g^k\gamma_n$ is central in $G/\gamma_n(G)$, then $[h,g^k]\gamma_n=\gamma_n$ for all $h\gamma_n\in G/\gamma_n(G)$, and thus, $[h_1,g^k,h_2]\gamma_{n+1}=\gamma_{n+1}$ for all $h_1\gamma_{n+1},h_2\gamma_{n+1}\in G/\gamma_{n+1}(G)$. Thus, for the induction case, we have:
    \begin{align*}
        [g^{k(l+1)},h]\gamma_{n+1} & =[g^{kl}g^k,h]\gamma_{n+1} \\
        & =[g^k,h][h,g^k,g^{kl}][g^{kl},h]\gamma_{n+1} \\
        & =[g^k,h][g^k,h^l]\gamma_{n+1} \\
        & =[g^k,h^{l+1}][h^l,g^k,h]^{-1}\gamma_{n+1} \\
        & =[g^k,h^{l+1}]\gamma_{n+1}.
    \end{align*}
    \end{proof}
\end{lemma}

\begin{theorem}\label{c-d=mdc}
If $|c-d|=\gcd(c,d)$, then the $R_{\infty}$-nilpotency degree of $G=\textnormal{loll}(a,b,c,d)$ is infinite.
\end{theorem}

\begin{proof}
Let $m=\gcd(c,d)=|c-d|$. Note that $y^{d-c}=[t,y^c]\in \gamma_2(G)$, so $y^m \in \gamma_2(G)$. Since $m$ divides $c$, it follows that $y^c\in \gamma_2(G)$. Therefore, $y^{d-c}=[t,y^c] \in [G,\gamma_2(G)]=\gamma_3(G)$, i.e. $y^m\in \gamma_3(G)$. Inductively, we conclude that $y^m\in \gamma_n(G)$ for all $n\geq 2$. Define the group

\[ \begin{array}{rl}
    H & = \langle \tilde{x}, \tilde{y}, \tilde{t} \: | \: \tilde{x}^a=\tilde{y}^b, \: \tilde{t}\tilde{y}^c\tilde{t}^{-1}=\tilde{y}^d, \: \tilde{y}^m=1\rangle \vspace{0.2cm}\\
     & =\langle \tilde{x}, \tilde{y}, \tilde{t} \: | \: \tilde{x}^a=\tilde{y}^b, \: \tilde{y}^m=1\rangle.
\end{array}\] Note that the two presentations above are equivalent because $m$ divides both $c$ and $d$, and hence $\tilde{y}^m=1$ implies $\tilde{y}^c=\tilde{y}^d=1$. We first claim that $G/\gamma_n(G)\simeq H/\gamma_n(H)$, for all $n\geq 2$. In fact, the map $f\colon G\to H/\gamma_n(H)$ defined by $x \mapsto \tilde{x}\gamma_n(H)$, $y \mapsto \tilde{y}\gamma_n(H)$, $t \mapsto \tilde{t}\gamma_n(H)$ is a group homomorphism and induces naturally $\overline{f}\colon G/\gamma_n(G)\to H/\gamma_n(H)$. Similarly, the map $g\colon H\to G/\gamma_n(G)$ with $\tilde{x}\mapsto x\gamma_n(G)$, $\tilde{y}\mapsto y\gamma_n(G)$ and $\tilde{t} \mapsto t\gamma_n(G)$ is a group homomorphism (because $y^m\in \gamma_n(G)$) and induces $\overline{g}\colon H/\gamma_n(H)\to G/\gamma_n(G)$. But $\overline{f}$ and $\overline{g}$ are clearly inverses to each other, so the claim is true.

Now we claim that, for all $n\geq 3$, the quotient $\overline{H}=H/\gamma_n(H)\simeq G/\gamma_n(G)$ does not possess $R_{\infty}$. This will finish the proof, since $G/\gamma_2(G)$ is a finitely generated abelian group and hence does not have $R_\infty$ (see the proof of Corollary \ref{atleast2}).
We focus on $\overline{H}$ and, for clarity, relabel the generators $\tilde{x}, \tilde{y}, \tilde{t}$ of $H$ simply by $x,y,t$. We use a similar strategy as in the proof of Theorem \ref{rank0}. Note that $t^{am}\gamma_3(H)$ is central in $H/\gamma_3(H)$ because $[t^{am},y]\gamma_3(H)=[t,y^{am}]\gamma_3(H)=\gamma_3(H)$ and $[t^{am},x]\gamma_3(H)=[t,x^{am}]\gamma_3(H)=[t,y^{bm}]\gamma_3(H)=\gamma_3(H)$. Thus, by Lemma \ref{comutatpotencomut}, $t^{(am)^2}\gamma_4(H)$ is central in $H/\gamma_4(H)$, since
\begin{align*}
& [t^{(am)^2},y]\gamma_4(H)=[t^{am},y^{am}]\gamma_4(H)=\gamma_4(H), \\
& [t^{(am)^2},x]\gamma_4(H)=[t^{am},x^{am}]\gamma_4(H)=[t^{am},y^{bm}]\gamma_4(H)=\gamma_4(H).
\end{align*}

With the same reasoning, one shows inductively that $t^{(am)^{n-2}}\gamma_n(H)$ is central in $H/\gamma_n(H)$. Therefore, $\overline{K}=\langle t^{(am)^{n-2}}\gamma_n(H)\rangle$ is a central subgroup of $\overline{H}$ which contains a power of each of its generators. This implies the quotient $\overline{H}/\overline{K}$ coincides with its torsion subgroup and must be finite, since it is finitely generated nilpotent. Consider the automorphism $\varphi\colon H\to H$
with $\varphi(x)=x^{-1}$, $\varphi(y)=y^{-1}$ and $\varphi(t)=t^{-1}$. Let $\overline{\varphi}\colon \overline{H}\to \overline{H}$ be its induced automorphism on $\overline{H}$. The subgroup $\overline{K}$ is invariant under $\overline{\varphi}$, so we have the following commutative diagram:

\[
\xymatrix{
  1 \ar[r] & \overline{K} \ar[d]_-{\overline{\varphi}|_{\overline{K}}} \ar[r]^-{} & \overline{H} \ar[d]_-{\overline{\varphi}} \ar[r]^-{} & \overline{H}/\overline{K} \ar[d]^-{\overline{\overline{\varphi}}} \ar[r] & 1 \\
  1 \ar[r] & \overline{K} \ar[r]_-{} & \overline{H} \ar[r]_-{} & \overline{H}/\overline{K} \ar[r] & 1
}
\]

Since $\overline{H}/\overline{K}$ is finite, it follows that $R(\overline{\overline{\varphi}})<\infty$. Furthermore, $\overline{K}$ is cyclic and $\overline{\varphi}|_{\overline{K}}$ acts as a multiplication by $-1$, so $R(\overline{\varphi}|_{\overline{K}})<\infty$. Since $\overline{K}$ is central in $\overline{H}$, it follows from Lemma \ref{cavelemma} item 2) that $R(\overline{\varphi})<\infty$. This completes the proof.
\end{proof}

From now on, we may assume $|c-d|>\gcd(c,d)$. In particular, we assume $|c-d|>1$, since otherwise, we would have $|c-d|=1=\gcd(c,d)$.

\section{Normal forms, $G/\gamma_3$ and a refinement of Theorem \ref{DaciDekimpe}}\label{normalformsection}

In this section, we continue to study $G=\textnormal{loll}(a,b,c,d)$. We obtain a presentation for the group $G/\gamma_3$ and use a normal form for its elements to obtain that the $R_{\infty}$-nilpotency degree of $G$ is $2$ in some cases, and also to refine Theorem \ref{DaciDekimpe}. Recall that we are assuming $a>1$ since the introduction.

\subsection*{Step 1: an identification} Denote by $F=F(x,y,t)$ the free group on three generators $x,y,t$. We have $G=F/M$, where $M=\langle\!\langle R \rangle\!\rangle$, $R=\{x^ay^{-b},ty^ct^{-1}y^{-d}\}\subset F$. Write $\overline{F}=F/\gamma_3(F)$ and let $f\colon F \to \overline{F}$ be the usual epimorphism. We claim that $G/\gamma_3(G)\simeq\overline{F}/\langle\!\langle f(R) \rangle\!\rangle$. In fact, first note that $\gamma_3(G)=\gamma_3(F/M)=\gamma_3(F)M/M$. Therefore, by the third Isomorphism Theorem,
\[
\frac{G}{\gamma_3(G)}=\frac{F/M}{\gamma_3(F)M/M}\simeq \frac{F}{\gamma_3(F)M}.
\]
Now, since $\gamma_3(F)\leq \gamma_3(F)M$ and both are normal subgroups of $F$, again by the third Isomorphism Theorem we obtain
\[
\frac{F}{\gamma_3(F)M}\simeq \frac{F/\gamma_3(F)}{\gamma_3(F)M/\gamma_3(F)}=\frac{\overline{F}}{\gamma_3(F)M/\gamma_3(F)},
\] and the rightmost denominator above equals the normal closure $\langle\!\langle f(R) \rangle\!\rangle$ in $\overline{F}$, since $f$ is an epimorphism. This shows our claim and allows us to identify $G/\gamma_3(G)=\overline{F}/\langle\!\langle f(R) \rangle\!\rangle$.

\subsection*{Step 2: a presentation} The group $\overline{F}=F/\gamma_3(F)$  is the free nilpotent group of rank $3$ and nilpotency class $2$. It is well-known to be a polycyclic group, and a standard presentation for it can be computed, for example, using \cite[Ch.~9, Sec.~10]{Sims}. We will consider the following equivalent presentation for $\overline{F}$, by a slight abuse of notation, using the same letters $x,y,t$ for the generators:

Generators: $x,y,t,u,v,w$;

Relations:
\begin{itemize}
\item $u=[x,y]$;
\item $v=[x,t]$;
\item $w=[y,t]$;
\item $u,v,w$\ \textnormal{are central}.
\end{itemize}

The last row, of course, is an abbreviation for the set of relations
\begin{align*}
& [u,x]=[u,y]=[u,t]=[u,v]=[u,w]=[v,x]= 1, \\
& [v,y]=[v,t]=[v,w]=[w,x]=[w,y]=[w,t]=1.
\end{align*}

Since $G/\gamma_3(G)=\overline{F}/\langle\!\langle f(R) \rangle\!\rangle$ and
\[
f(ty^ct^{-1}y^{-d})=ty^ct^{-1}y^{-d}=[t,y^c]y^{c-d}=[t,y]^cy^{c-d}=y^{c-d}w^{-c},
\]
a possible presentation for $G/\gamma_3(G)$ is:

Generators: $x,y,t,u,v,w$;

Relations:
\begin{itemize}
\item $x^a=y^b$;
\item $y^{c-d}=w^c$;
\item $u=[x,y]$;
\item $v=[x,t]$;
\item $w=[y,t]$;
\item $u,v,w$\ \textnormal{are central}.
\end{itemize}

The above presentation, although correct, is not algorithmically useful to compute the normal form of elements in $G/\gamma_3(G)$. To make it more useful, we will perform the Tietze transformations of adding to the list of relations the following three consequences, which hold in $G/\gamma_3(G)$ thanks to Proposition \ref{commut3}:

\begin{itemize}
    \item $u^n=1$ for $n=\gcd(a,b,d-c)$. Indeed, $u^a=[x,y]^a=[x^a,y]=[y^b,y]=1$, $u^b=[x,y]^b=[x,y^b]=[x,x^a]=1$ and $u^{c-d}=[x,y]^{c-d}=[x,y^{c-d}]=[x,w^c]=1$;
    \item $v^a=w^b$. Indeed, $v^a=[x,t]^a=[x^a,t]=[y^b,t]=[y,t]^b=w^b$.
    \item $w^{c-d}=1$. Indeed, $w^{c-d}=[y,t]^{c-d}=[y^{c-d},t]=[w^c,t]=1$.
\end{itemize}

Hence, we will use the following presentation for $G/\gamma_3(G)$:

Generators: $x,y,t,u,v,w$;

Relations:
\begin{itemize}
\item $x^a=y^b$;
\item $y^{c-d}=w^c$;
\item $u^n=1$ for $n=\gcd(a,b,c-d)$;
\item $v^a=w^b$;
\item $w^{c-d}=1$;
\item $u=[x,y]$;
\item $v=[x,t]$;
\item $w=[y,t]$;
\item $u,v,w$\ \textnormal{are central}.
\end{itemize}

\subsection*{Step 3: normal form and uniqueness.} At this step, we recommend that readers who are not familiar with polycyclic presentations consult \cite[Ch.~9]{Sims}. The presentation given above to the free nilpotent group $\overline{F}=F/\gamma_3(F)$ is called a \emph{consistent polycyclic presentation}. This means that there is uniqueness in the normal form of elements, given as follows: for every $g \in \overline{F}$, there are unique integers $k_i$, $1 \leq i \leq 6$, such that
\[
g=x^{k_1}y^{k_2}t^{k_3}u^{k_4}v^{k_5}w^{k_6}.
\] Denote by $N$ the normal closure $N= \langle\!\langle f(R) \rangle\!\rangle=\langle\!\langle x^ay^{-b},y^{c-d}w^{-c} \rangle\!\rangle \lhd \overline{F}$. In this context, we say that $N$ is \emph{normally generated} by the set $f(R)$ and call its elements the \emph{normal generators} of $N$. A normal form for elements in $G/\gamma_3(G)=\overline{F}/N$ involves certain positive integers called \emph{relative orders}, as we describe now. In $G/\gamma_3(G)$, denote for a moment $(x,y,t,u,v,w,1)=(g_1,g_2,g_3,g_4,g_5,g_6,g_7)$ and let $G_i$ be the subgroup generated by $g_i,\dots ,g_7$ for $1 \leq i \leq 7$. By the presentation of $G/\gamma_3(G)$, one can easily check that $G_i \rhd G_{i+1}$ for $1 \leq i \leq 6$. Since $G_1=G/\gamma_3(G)$ and $G_7=1$, this is a polycyclic series for $G/\gamma_3(G)$. The relative order of $g_i$ is defined to be the order of the cyclic group $G_i/G_{i+1}=\langle g_i G_{i+1}\rangle$; equivalently, it is the least positive integer $m_i$ such that $g_i^{m_i} \in G_{i+1}$ if $G_i/G_{i+1}$ is finite, and infinite otherwise.

Next, we use this terminology to describe a normal form for elements in $G/\gamma_3(G)$.

\begin{lemma}\label{normalformlemma}
For every $g \in G/\gamma_3(G)$, there are unique integers $k_i$, $1 \leq i \leq 6$, such that $k_1,k_5 \in [0,a)$, $k_2,k_6 \in [0,|c-d|)$, $k_4 \in [0,n)$ and
\[
g=x^{k_1}y^{k_2}t^{k_3}u^{k_4}v^{k_5}w^{k_6}.
\]    
\end{lemma}
\begin{proof}
We follow the language and procedure on Sims's work \cite[Sec.~9.5, 9.6]{Sims}. Existence of the normal form above follows easily from the existence of normal form for elements in $\overline{F}$ and the successive use of the \emph{power relations} of $G/\gamma_3(G)$, which are $x^a=y^b$, $y^{c-d}=w^c$, $u^n=1$, $v^a=w^b$ and $w^{c-d}=1$, in this exact order, in an algorithm called \emph{collection}. By Chapter 9 of \cite{Sims}, showing uniqueness is equivalent to showing that the relative order of $x,y,u,v$ and $w$ is, respectively, $a,|c-d|,n,a$ and $|c-d|$, and that the relative order of $t$ is infinite. The latter claim follows easily, for example, from the fact that $t$ has infinite order in the abelian group $G/\gamma_2(G)$, which is a quotient of $G/\gamma_3(G)$, by the third Isomorphism Theorem. To determine the other relative orders, we first briefly describe an explicit procedure, based on \cite{Sims}, to compute relative orders in the quotient group $\overline{F}/N$ of $\overline{F}$. The procedure below works, in fact, for any finitely generated free nilpotent group and any normal subgroup. 
\begin{enumerate}
    \item Denote the standard list of generators of $\overline{F}$ by $(f_1,\ldots,f_n)$ (the ordering is relevant). In our case, $(f_1,f_2,f_3,f_4,f_5,f_6)=(x,y,t,u,v,w)$;
    \item Find a finite list $U=(u_1,\ldots,u_s)$ of elements of $N$ such that $N=\langle u_1,\ldots,u_s\rangle$. In our case, this involves computing the conjugates of the form $f_imf_i^{-1}$ and $f_i^{-1}mf_i$ for each normal generator $m$ of $N$ and each generator $f_i$ of $\overline{F}$, in order to determine which elements need to be added to the list of normal generators to obtain such a list $U$.
    \item For each $1 \leq i \leq s$, write $u_i=f_1^{k_{i1}}\cdots f_n^{k_{in}}$, where $k_{ij} \in \mathbb{Z}$, $1 \leq j \leq n$. Consider the $s \times n$ matrix $A=(k_{ij})$, called the \emph{exponent matrix} associated to $U$ in Sims's terminology.
    \item Perform the \emph{row reduction} process to $A$. The three allowed row operations are: interchanging two rows, multiplying a row by $-1$ and adding an integral multiple of one row to another one. After suitable row operations and deleting all zero rows (if any), we obtain an $r \times n$ matrix $\tilde{A}=(\alpha_{ij})$, for some $1 \leq r \leq \min\{s,n\}$, with no zero rows and with the following property: if the first nonzero entry of the $i$-th row is $\alpha_{ij_i}$ (these are called the \emph{corner entries} of $\tilde{A}$), then $\alpha_{ij_i}>0$ and $j_1<j_2<\cdots j_r$. In other words, corner entries are positive and appear in columns progressively further to the right.
    \item Then, by Sims's work \cite{Sims}, especially Sections 9.5 and 9.6, the relative orders of the image of the generators $f_{j_i}$ in the quotient group $\overline{F}/N$ are exactly the respective corner entries $\alpha_{ij_i}$, while the relative orders of the images of the other generators $f_j$ (if any) are infinite.
    
\end{enumerate}

In our particular case, we claim that $N=\langle U \rangle$ for $U=\{ x^ay^{-b},y^{c-d}w^{-c}, u^n,v^aw^{-b},w^{c-d} \}$. By the previous subsection, the generators $u^n,v^aw^{-b}$ and $w^{c-d}$ must belong to $N$, since relations $u^n=1,v^a=w^{b}$ and $w^{c-d}=1$ hold on $G/\gamma_3(G)=\overline{F}/N$. Therefore, $\langle U \rangle \subset N$. To show that $N \subset \langle U \rangle$, it suffices to show that $\langle U \rangle$ is normal in $\overline{F}$ and contains the images under $f$ of the
defining relators $x^ay^{-b}$ and $y^{c-d}w^{-c}$. The latter is immediate from the definition of $U$. Since $u,v$ and $w$ are central elements in $\overline{F}$, it suffices to show that, for every $g \in \{x,x^{-1},y,y^{-1},t,t^{-1}\}$, the elements $g(x^ay^{-b})g^{-1}$ and $g(y^{c-d}w^{-c})g^{-1}$ belong to $\langle U \rangle$. Let $\epsilon=\pm1$ for a moment. Since $n$ divides $a$, $b$ and $c-d$, we have $u^a,u^b,u^{c-d} \in \langle U \rangle$. Thus,
\begin{align*}
x^\epsilon(x^ay^{-b})x^{-\epsilon} 
&= x^ax^\epsilon y^{-b} x^{-\epsilon}=x^a[x^\epsilon,y^{-b}]y^{-b}=x^au^{-\epsilon b}y^{-b}=x^ay^{-b}u^{-\epsilon b} \in \langle U \rangle, \\
y^\epsilon(x^ay^{-b})y^{-\epsilon} 
&= y^\epsilon x^ay^{-\epsilon} y^{-b}=[y^\epsilon,x^{a}]x^ay^{-b}=u^{-\epsilon a} x^ay^{-b}\in \langle U \rangle, \\
t^\epsilon(x^ay^{-b})t^{-\epsilon} 
&= t^{\epsilon}x^at^{-\epsilon}t^{\epsilon}y^{-b}t^{-\epsilon}=[t^\epsilon,x^a]x^a[t^\epsilon,y^{-b}]y^{-b}=v^{-\epsilon a}x^a w^{\epsilon b}y^{-b}= (v^a w^{-b})^{-\epsilon}x^a y^{-b}\in \langle U \rangle,\\
x^\epsilon(y^{c-d}w^{-c})x^{-\epsilon}
&=  x^\epsilon y^{c-d}x^{-\epsilon}w^{-c}=[x^\epsilon,y^{c-d}]y^{c-d}w^{-c}=u^{\epsilon(c-d)}y^{c-d}w^{-c} \in \langle U \rangle,\\
y^\epsilon(y^{c-d}w^{-c})y^{-\epsilon}
&=  y^\epsilon y^{c-d}y^{-\epsilon}w^{-c}=y^{c-d}w^{-c} \in \langle U \rangle,\\
t^\epsilon(y^{c-d}w^{-c})t^{-\epsilon}
&= t^\epsilon y^{c-d}t^{-\epsilon}w^{-c}=[t^\epsilon,y^{c-d}]y^{c-d}w^{-c}= w^{-\epsilon(c-d)}y^{c-d}w^{-c}\in \langle U \rangle,
\end{align*} which shows the claim.
The exponent matrix associated to $U$ is
\[
A=
\begin{pmatrix}
a & -b  & 0 & 0 & 0 & 0    \\
0 & c-d & 0 & 0 & 0 & -c   \\
0 & 0   & 0 & n & 0 & 0    \\
0 & 0   & 0 & 0 & a & -b   \\
0 & 0   & 0 & 0 & 0 & c-d
\end{pmatrix}.
\]For any integer $l$, denote by $l \bmod{|c-d|}$ the remainder of the euclidean division of $l$ by $|c-d|$. If we perform row reduction on $A$, we obtain
\[
\tilde{A}=
\begin{pmatrix}
a & -b\bmod{|c-d|} & 0 & 0 & 0 & -k\epsilon c \bmod{|c-d|}   \\
0 & |c-d| & 0 & 0 & 0 & -\epsilon c \bmod{|c-d|}   \\
0 & 0   & 0 & n & 0 & 0   \\
0 & 0   & 0 & 0 & a & -b \bmod{|c-d|}  \\
0 & 0   & 0 & 0 & 0 & |c-d|
\end{pmatrix},
\] where $k$ is such that $0 \leq -b +k|c-d|<|c-d|$ and $\epsilon=\textnormal{sgn}(c-d)$ is the sign of the nonzero integer $c-d$.
Then, the relative orders of $x,y,u,v$ and $w$ in the group $\overline{F}/N=G/\gamma_3(G)$ are the \emph{corner entries} of $\tilde{A}$, which are, respectively, $a,|c-d|,n,a$ and $|c-d|$, as desired.
\end{proof}

Now, we state and prove the main results of the section.

\begin{theorem}\label{pirulitosindice2}
If $G=\textnormal{loll}(a,b,c,d)$ satisfies $\gcd(a,c-d)=1=\gcd(c,d)$, and the number $1<|c-d|$ is odd, then the $R_{\infty}$-nilpotency degree of $G$ is $2$.
\end{theorem}

\begin{proof}
By Corollary \ref{atleast2}, it suffices to show that $Q=G/\gamma_3$ possesses property $R_\infty$. We use the presentation for $Q$ displayed in ``Step 2 - a presentation'':
\begin{align*}
Q = \langle x,y,t,u,v,w\ | \ &x^a = y^b, y^{c-d} = w^c, u^n = 1, v^a = w^b, w^{c-d} = 1,\\ &u = [x, y], v = [x, t], w = [y, t], u,v,w\textnormal{ central}\rangle,
\end{align*} where $n = \gcd(a,b,c-d)$, together with the normal form and relative orders established in Lemma \ref{normalformlemma}. Under the hypotheses of the present theorem, $u$ is trivial because $n=1$. Thus Lemma \ref{normalformlemma} gives a normal form
\[
g = x^{k_1}y^{k_2}t^{k_3}v^{k_5}w^{k_6}
\] for every $g \in Q$, with bounded exponents on $x,y,v,w$ and arbitrary exponent on $t$. Hence
$\langle x,y,v,w\rangle$ is finite, while $Q/\langle x, y, v,w\rangle$ is infinite cyclic, generated by the image of $t$. Therefore $\tau (Q) = \langle x,y,v,w \rangle$. Moreover, since $Q$ is generated by $x, y, t$ and has nilpotency class at most $2$, its commutator subgroup is generated by $u = [x, y], v = [x, t]$ and $w = [y, t]$. Since $u = 1$, we obtain $\gamma_2(Q) = \gamma_2(G)/\gamma_3(G) = \langle v,w\rangle$. Both $\tau(Q)$ and $\gamma_2(Q)$ are characteristic subgroups of Q.

Let $\varphi \in Aut(G/\gamma_3)$. We have the following commutative diagram:

\[
\xymatrix{
  1 \ar[r] & \gamma_2/\gamma_3 \ar[d]_-{\varphi'} \ar[r]^-{} & G/\gamma_3 \ar[d]_-{\varphi} \ar[r]^-{} &  G/\gamma_2 \ar[d]^-{\overline{\varphi}} \ar[r] & 1 \\
  1 \ar[r] & \gamma_2/\gamma_3 \ar[r]_-{} & G/\gamma_3 \ar[r]_-{} & G/\gamma_2 \ar[r] & 1
}\] Now, we use the normal form for elements in $G/\gamma_3$ and the fact that the abelian subgroups $\gamma_2/\gamma_3=\langle v,w \rangle$  and $\tau(G/\gamma_3)=\langle x,y,v,w\rangle$ are characteristic to write
\begin{eqnarray}\label{phi.1}
\varphi(x) &=& x^{k_{11}}y^{k_{12}}v^{k_{14}}w^{k_{15}},\\
\varphi(y) &=& x^{k_{21}}y^{k_{22}}v^{k_{24}}w^{k_{25}},\\
\varphi(t) &=& x^{k_{31}}y^{k_{32}}t^{k_{33}}v^{k_{34}}w^{k_{35}},\\
\varphi(v) &=& v^{k_{44}}w^{k_{45}},\\
\label{phi.5} \varphi(w) &=& v^{k_{54}}w^{k_{55}},
\end{eqnarray} for $k_{i1},k_{i4} \in [0,a)$ and $k_{i2},k_{i5} \in [0,|c-d|)$. Furthermore, we have $k_{33}=\pm 1$. In fact, the abelianization $A=G/\gamma_2(G)$ can be presented as
\[
A= \langle x, y, t \ |\  x^a = y^b, y^{c-d} = 1, [x, y] = [x, t] = [y, t] = 1\rangle.
\]
Note that no defining relation above involves $t$ except commutativity, so the image of $t$ gives an infinite cyclic summand modulo torsion in $A$. Moreover, $y$ is torsion because
$y^{c-d} = 1$ in $A$, and $x$ is torsion because $x^a = y^b$ and $y$ is torsion. Thus $\tau(A) = \langle x, y\rangle$, and $A/\tau(A) \simeq \langle t \rangle \simeq \mathbb{Z}$. In this infinite cyclic quotient, the induced automorphism $\overline{\overline{\varphi}}$ maps $t$ to $t^{k_{33}}$. Hence, $k_{33}=\pm 1$.

We will show in a moment that $k_{33}=1$. But first, we claim that this fact implies $R(\varphi)=\infty$. Indeed, if $k_{33}=1$, then $\overline{\overline{\varphi}}$ is the identity map on $\mathbb{Z} \simeq \langle t \rangle \simeq A/\tau(A)$, so $R(\overline{\overline{\varphi}})=\infty$. Since
$\tau(A)$ is characteristic in $A$, we have $R(\overline{\varphi})=\infty $ by Lemma \ref{cavelemma} part (1), which in turn implies $R(\varphi)=\infty$, again by Lemma \ref{cavelemma} part (1).

Now we prove that $k_{33}=1$. Let us prove that
\begin{equation}\label{equation1}
x^{(c-d)k_{21}}y^{(c-d)k_{22}}v^{(c-d)k_{24}}w^{(c-d)k_{25}} =v^{ck_{21}k_{33}}w^{ck_{22}k_{33}}.
\end{equation} Indeed, consider the relations $y^{c-d}=w^c$ and $w=[y,t]$. Applying $\varphi$ to both and combining them we obtain
\begin{equation}\label{equation2}
\varphi(y)^{c-d} = \varphi(w)^c=\varphi([y,t])^c=[\varphi(y),\varphi(t)]^c.
\end{equation}
Now, using the displayed forms of $\varphi(y)$ and $\varphi(t)$, and using that $u = [x, y] = 1$ and that $v,w$ are central, we have
\begin{align*}
[\varphi(y), \varphi(t)] &= [x^{k_{21}}y^{k_{22}}v^{k_{24}}w^{k_{25}} , x^{k_{31}}y^{k_{32}}t^{k_{33}}v^{k_{34}}w^{k_{35}}]\\
&= [x^{k_{21}}y^{k_{22}},t^{k_{33}}]\\
&= [x,t]^{k_{21}k_{33}}[y,t]^{k_{22}k_{33}} \\
&= v^{k_{21}k_{33}}w^{k_{22}k_{33}}.
\end{align*} Here we used the identities $[\alpha\beta, \gamma] = [\alpha, \gamma][\beta, \gamma]$ and $[\alpha^r, \gamma^s] = [\alpha, \gamma]^{rs}$ of the class-2 nilpotent group $G/\gamma_3$, which are direct consequences of Proposition \ref{commut3}. Therefore, the right-hand side of Equation \ref{equation2} is
\[
[\varphi(y), \varphi(t)]^c = (v^{k_{21}k_{33}}w^{k_{22}k_{33}})^c= v^{ck_{21}k_{33}}w^{ck_{22}k_{33}}.
\] 
The fact that $u = [x, y] = 1$ and that $v,w$ are central also implies that the subgroup $\langle x,y,v,w\rangle$ is abelian. Hence,
the left-hand side of Equation \ref{equation2} is
\[
\varphi(y)^{c-d}=(x^{k_{21}}y^{k_{22}}v^{k_{24}}w^{k_{25}})^{c-d}=x^{(c-d)k_{21}}y^{(c-d)k_{22}}v^{(c-d)k_{24}}w^{(c-d)k_{25}},
\]so Equation \ref{equation1} is true. Since $1<a$, and hence $x\neq 1$, by uniqueness of normal form Equation~\ref{equation1} implies $(c-d)k_{21}\equiv 0 \pmod{a}$. Since $\gcd(a,c-d)=1$, it follows that $k_{21}\equiv0 \pmod{a}$. Because $0 \leq k_{21}<a$, we obtain $k_{21}=0$. Hence, 
\[
y^{(c-d)k_{22}}v^{(c-d)k_{24}}w^{(c-d)k_{25}} = w^{ck_{22}k_{33}}.
\] By using $w^{c-d}=1$ and $y^{c-d}=w^c$, we obtain
\[
v^{(c-d)k_{24}}w^{ck_{22}} = w^{ck_{22}k_{33}}.
\] We proceed in a similar manner to the argument above. Since $v \neq 1$ (because $1<|c-d|$), by uniqueness of normal form, $(c-d)k_{24} \equiv 0 \pmod a$. Since $\gcd(a,c-d)=1$ and $0 \leq k_{24} < a$, we obtain $k_{24} = 0$, which gives us $w^{ck_{22}} = w^{ck_{22}k_{33}}$. Because $w \neq 1$ (for $1<|c-d|$), it follows that
\[
ck_{22}(1-k_{33})\equiv0\pmod{c-d}.
\] Finally, suppose by contradiction that $k_{33}=-1$. Then, $2ck_{22}\equiv0 \pmod{c-d}$. Since $c-d$ is odd and $\gcd(c,c-d)=1$, the numbers $2$ and $c$ are invertible modulo $c-d$. Hence, $k_{22}\equiv0\pmod{c-d}$, which implies $k_{22}=0$. This contradicts the fact that the automorphism $\overline{\varphi}$ maps $y$ to the element $y^{k_{22}}$, which must be nontrivial. It follows that $k_{33}=1$, as desired.
\end{proof}

\begin{theorem}\label{notdivideb}
If $G=\textnormal{loll}(a,b,c,d)$ satisfies $\gcd(a,c-d)=1=\gcd(a,b)$, $|c-d|>\gcd(c,d)$ and $c-d$ does not divide $b$, then the $R_{\infty}$-nilpotency degree of $G$ is $2$.
\end{theorem}

\begin{proof} We proceed similarly to the proof of Theorem \ref{pirulitosindice2}. By Corollary \ref{atleast2}, it suffices to show that $G/\gamma_3$ possesses property $R_\infty$. We have $u=1$ as well, and if $\varphi \in Aut(G/\gamma_3)$, it can be described by equations (\ref{phi.1})--(\ref{phi.5}). We will show that $k_{33}= 1$, which will complete the proof.

Suppose by contradiction that $k_{33}=-1$. From relation $v=[x,t]$ in $G/\gamma_3$, we obtain:

\begin{align*}
    \varphi(v) & =[\varphi(x),\varphi(t)] \\
    v^{k_{44}}w^{k_{45}} & =[x^{k_{11}}y^{k_{12}},x^{k_{31}}y^{k_{32}}t^{-1}] \\
    v^{k_{44}}w^{k_{45}} & =v^{-k_{11}}w^{-k_{12}}.
\end{align*}
Thus, $k_{44}\equiv -k_{11} \pmod a$ and $k_{45} \equiv -k_{12} \pmod{c-d}$. In a similar way, from relation $w=[y,t]$, we obtain
$k_{55} \equiv -k_{22} \pmod{c-d}$. We now claim that $ak_{12}+b(k_{11}-k_{22}) \equiv 0 \pmod{c-d}$. Indeed, from relation $x^a=y^b$ and the fact that $\tau(G/\gamma_3)=\langle x,y,v,w \rangle$ is abelian, we obtain
\begin{align*}
    \varphi(x)^a & =\varphi(y)^b \\
    (x^{k_{11}}y^{k_{12}}v^{k_{14}}w^{k_{15}})^a & = (x^{k_{21}}y^{k_{22}}v^{k_{24}}w^{k_{25}})^b\\
    x^{ak_{11}}y^{ak_{12}}v^{ak_{14}}w^{ak_{15}} & = x^{bk_{21}}y^{bk_{22}}v^{bk_{24}}w^{bk_{25}}\\
    y^{ak_{12}+bk_{11}}w^{ak_{15}+bk_{14}} & = x^{bk_{21}}y^{bk_{22}}v^{bk_{24}}w^{bk_{25}}.
\end{align*}By uniqueness of normal form, we have $bk_{21} \equiv 0 \pmod a$. Since $\gcd(a,b)=1$, we get $k_{21} \equiv 0 \pmod a$ and then $k_{21}=0$, since $0 \leq k_{21}<a$. By manipulating the above equation we then obtain
\[
y^{ak_{12}+b(k_{11}-k_{22})}= v^{bk_{24}}w^{b(k_{25}-k_{14})-ak_{15}},
\] thus $ak_{12}+b(k_{11}-k_{22}) \equiv 0 \pmod{c-d}$, again by uniqueness of normal form. Applying a similar argument to relation $v^a=w^b$, we obtain $ak_{45}+b(k_{44}-k_{55}) \equiv 0 \pmod{c-d}$. Using all of the equations obtained so far, we get:
\begin{align*}
     -ak_{45}& \equiv ak_{12} \pmod{c-d}\\
    bk_{44}-bk_{55}& \equiv ak_{12} \pmod{c-d}\\
    bk_{44}+bk_{22}& \equiv ak_{12}\pmod{c-d}\\
    bk_{44}+(ak_{12}+bk_{11})& \equiv ak_{12} \pmod{c-d}\\
    b(k_{11}+k_{44})& \equiv 0 \pmod{c-d}.
\end{align*} Since $k_{44}+k_{11}=0 \pmod a$ and $0 \leq k_{44}+k_{11}<2a$, we must have $k_{11}+k_{44}\in \{0,a\}$. We will come to a contradiction in both cases.

If $k_{11}+k_{44}=0$, it means that $k_{11}=k_{44}=0$, and since $k_{21}=0$ (for the same reason as in the previous theorem), the induced automorphism on $\tau(G/\gamma_2)$ preserves $\langle y \rangle$. In this way, $\varphi$ induces an automorphism on $\tau(G/\gamma_2)/\langle y \rangle \simeq \langle x \rangle \simeq \mathbb{Z}_{a}$, given by $\varphi(x)=x^{k_{11}}$, but since $k_{11}=0$, we have a contradiction.

If $k_{11}+k_{44}=a$, this means that $ba=b(k_{11}+k_{44}) \equiv 0 \pmod{c-d}$. Since $\gcd(a,c-d)=1$, it follows that $c-d$ divides $b$, a contradiction.

Therefore, we must have $k_{33}=1$ and $R(\varphi)=\infty$, as desired.
\end{proof}

The previous result allows us to refine Theorem \ref{DaciDekimpe}. In the case $n-m=2\gcd(m,n)$ (the fifth of six cases), there is not an exact computation of the $R_{\infty}$-nilpotency degree $r$ of $BS(m,n)$, but only an upper bound $2 \leq r \leq p+2$. Below, we compute $r$ in three quarters of this case, assuming only that $m,n$ are not both even. So, we may call it a ``12.5\%'' refinement.

\begin{theorem}[A 12.5\% refinement of Theorem \ref{DaciDekimpe}]\label{refinement}
Let $0<m\leq |n|$ with $m\neq n$ and suppose $d=\gcd(m,n)>1$. If $n-m=2d$ and $m,n$ are not both even, then the $R_{\infty}$-nilpotency degree of $BS(m,n)$ is $2$.
\end{theorem}

\begin{proof}
Let $G=\textnormal{loll}(d,1,\frac{m}{d},\frac{n}{d})$. Note that $G \simeq BS(m,n)$, via the Tietze transformation \cite[p.~89]{LyndonSchupp} of replacing $y$ by $x^d$. Now, one can see that $G$ satisfies the requirements of Theorem \ref{notdivideb}. Therefore, the $R_{\infty}$-nilpotency degree of $BS(m,n)$ is $2$ as desired.
\end{proof}

\section{Sylow subgroups and lollipop GBS groups}\label{Sylow}

In this section, we use well-known facts about Sylow subgroups and their relation with finite nilpotent groups to compute the $R_\infty$-nilpotency degree of some previously untreated cases of lollipop GBS groups. We also use the normal forms from Section \ref{normalformsection}. Let us begin by recalling some well-known facts about nilpotent groups and Sylow subgroups (see e.g. \cite[Chapter II]{Hungerford}). 

\begin{lemma}\label{lema1}
A finite group $G$ is nilpotent if and only if all of its Sylow subgroups are normal. Equivalently, $G$ is nilpotent if and only if $G$ is the direct product of its Sylow subgroups.
\end{lemma}

\begin{lemma}\label{lema2}
A normal Sylow subgroup is always a characteristic subgroup.
\end{lemma}

Recall that a \emph{Hall divisor} of an integer $n$ is defined as a divisor $m$ of $n$ such that $m$ and $n/m$ are coprime. In other words, $m$ is either $1$ or a product of maximal prime powers in the prime decomposition of $n$. For an integer $k\geq 1$, let $P(k)$ denote the set of prime numbers appearing in the prime decomposition of $k$. In particular, $P(1)=\emptyset$. The next lemma follows directly from these definitions.

\begin{lemma}\label{lemanumerodesylow}
If $m$ is a Hall divisor of $n$ and $k \geq 2$ is such that $P(k)\cap P(n) \subset P(m)$, then $km$ is a Hall divisor of $kn$.
\end{lemma}

\begin{theorem}\label{Sylow1}
Let $G=\textnormal{loll}(a,b,c,d)$. Suppose that $n=\gcd(a,b,c-d)=1$, that $m=\gcd(c,d)$ is a Hall divisor of $|c-d|$ and that $P(a)\cap P(|c-d|) \subset P(m)$. Then, the $R_{\infty}$-nilpotency degree of $G$ is $2$ in any of the following cases:
\begin{itemize}
    \item $d<0$ and $d \neq -c$;
    \item $d>0$ and $d-c \geq 3m$.
\end{itemize}
\end{theorem}

\begin{proof}
By Corollary \ref{atleast2}, it suffices to show that $G/\gamma_3(G)$ possesses property $R_\infty$. Let $a'=a/\gcd(a,b)$, $b'=b/\gcd(a,b)$, $c'=c/\gcd(c,d)$ and $d'=d/\gcd(c,d)$. Consider the presentation $BS=BS(c',d')=\left<z,s\ |\ sz^{c'}s^{-1}=z^{d'} \right>$. It is straightforward to verify that $\varphi \colon G \to BS(c',d')$ given by $\varphi(x)=z^{b'}$, $\varphi(y)=z^{a'}$ and $\varphi(t)=s$ is a well-defined homomorphism. Since $\gcd(a',b')=1$, it follows that $\varphi$ is surjective. Write $\overline{G}=G/\gamma_3(G)$ and \[
\overline{BS}=BS(c',d')/\gamma_3(BS(c',d')),\] and denote by $\varphi\colon \overline{G} \to \overline{BS}$ the induced epimorphism. 




We consider the presentation of $\overline{G}$ as in Section \ref{normalformsection} and, using a similar argument, we construct a presentation for $\overline{BS}$. Let $F = F(z,s)$ be the free group on two generators $z$ and $s$. Let $\overline{F} = F/\gamma_3(F)$, $f\colon F \to \overline{F}$ be the natural epimorphism, and set $R = \{\, sz^{c'}s^{-1}z^{-d'} \,\}$. By the same argument used in Step~1 of Section \ref{normalformsection}, it follows that
\[
\overline{BS} \cong \overline{F}/\langle\!\langle f(R)\rangle\!\rangle.
\]

Following Step~2, a presentation for $\overline{F}$ is
\[
\overline{F} = \langle z,s,r \mid r = [z,s],\ [r,z] = 1,\ [r,s] = 1 \rangle.
\]
Moreover,
\[
f(sz^{c'}s^{-1}z^{-d'}) = sz^{c'}s^{-1}z^{-d'} = z^{c'-d'}r^{-c'},
\]
and therefore
\[
\overline{F}/\langle\!\langle f(R)\rangle\!\rangle
= \langle z,s,r \mid r = [z,s],\ [r,z] = 1,\ [r,s] = 1,\ z^{c'-d'}r^{-c'} = 1 \rangle.
\]

Although this presentation is correct, we may simplify it via a Tietze transformation by adding a relation that is a consequence of the existing ones. Indeed,
\[
r^{c'-d'} = [z,s]^{c'-d'} = [z^{c'-d'},s] = [r^{c'},s] = 1.
\]
Hence, we obtain the presentation
\[
\overline{BS} = \langle z,s,r \mid r = [z,s],\ [r,z] = 1,\ [r,s] = 1,\ z^{c'-d'} = r^{c'},\ r^{c'-d'} = 1 \rangle.
\] By following Step 3 of Section \ref{normalformsection}, we provide next a sketch to show uniqueness of the following normal form for elements in $\overline{BS}$: for every $g \in \overline{BS}$, there are unique integers $k_i$, $1 \leq i \leq 3$, such that $k_1,k_3 \in [0,|c'-d'|)$ and
\[
g=z^{k_1}s^{k_2}r^{k_3}.
\] In fact, the relative order of $s$ is infinite, since $s$ has infinite order in the abelianized group $BS/\gamma_2(BS)$, which is a quotient of $\overline{BS}$. Let $U=\{z^{c'-d'}r^{-c'},r^{c'-d'}\}\leq \overline{F}$. Since $r^{c'-d'}=1$ in $\overline{BS}=\overline{F}/\langle\!\langle f(R)\rangle\!\rangle=\overline{F}/\langle\!\langle z^{c'-d'}r^{-c'}\rangle\!\rangle$, it follows that $\langle\!\langle z^{c'-d'}r^{-c'} \rangle\!\rangle \supset \langle U \rangle$. The opposite inclusion holds because the right subgroup is normal in $\overline{F}$, since
\[
(z^{k_1}s^{k_2}r^{k_3})(z^{c'-d'}r^{-c'})(r^{-k_3}s^{-k_2}z^{-k_1})=(z^{c'-d'}r^{-c'})r^{-k_2(c'-d')} \in \langle U \rangle,
\] for any $k_i$, $1 \leq i \leq 3$. The exponent matrix associated to $U$ is
\[
A=
\begin{pmatrix}
c'-d' & 0  & -c'     \\
0 & 0 & c'-d'    
\end{pmatrix},
\]which can be row reduced to
\[
\tilde{A}=
\begin{pmatrix}
|c'-d'| & 0 & -\epsilon c' \bmod{|c'-d'|} \\
0 & 0 & |c'-d'| 
\end{pmatrix},
\] where $\epsilon=\textnormal{sgn}(c'-d')$ is the sign of the nonzero integer $c'-d'$.
Then, by \cite[Ch.~9, Sec.~6]{Sims}, as discussed in Section \ref{normalformsection}, the relative orders of $z$ and $r$ in the group $\overline{BS}=\overline{F}/\langle\!\langle f(R)\rangle\!\rangle$ are the \emph{corner entries} of $\tilde{A}$, which are both $|c'-d'|$. This shows uniqueness of the normal form.

Recall that we proved in Section \ref{normalformsection} that for $u=[x,y]\gamma_3(G)$, $u^a=u^b=u^{d-c}=1$. Since $n=\gcd(a,b,c-d)=1$ in our case, we conclude that $u=1$ in $\overline{G}$. We claim that $\tau(\overline{G})=\left< x,y,v,w \right>$. In fact, since $u=1$, the subgroup $\left<x,y,v,w\right>$ is abelian and each generator has finite order, thus $\left<x,y,v,w\right> \subset \tau(\overline{G})$. Following the notations obtained in Section 4, if 
\[ x^{k_1}y^{k_2}t^{k_3}v^{k_5}w^{k_6}\in \tau(\overline{G}),\]
then $(x^{k_1}y^{k_2}t^{k_3}v^{k_5}w^{k_6})^l=1$ for some $l\geq 1$. By projecting naturally this equation on $G/\gamma_2=\langle x,y \rangle \oplus \langle t\rangle$ we obtain
\[x^{lk_1}y^{lk_2}t^{lk_3}=1.\]
Since $t$ has infinite order in $G/\gamma_2$ and $l \geq 1$ we obtain $k_3=0$.

Therefore, $\tau(\overline{G})=\left< x,y,v,w \right>$, and it is a finite nilpotent (in fact, abelian) group. Thus, by Lemmas \ref{lema1} and \ref{lema2}, $\tau(\overline{G})$ is the direct product of its Sylow subgroups, which are normal and characteristic. Because $\varphi(t)=s$ is of infinite order, it follows that $k_3\neq 0$ implies $x^{k_1}y^{k_2}t^{k_3}v^{k_5}w^{k_6} \not\in \ker(\varphi)$. Hence, $\ker(\varphi) \subset \tau(\overline{G})$ and it is also a finite abelian group. In particular, $\ker(\varphi)=\ker(\varphi)\cap \tau(\overline{G})=\ker(\varphi_\tau)$, where $\varphi_\tau\colon \tau(\overline{G}) \to \tau(\overline{BS})$ is the restricted torsion homomorphism. We claim that $\varphi_{\tau}$ is an epimorphism. Indeed, by following arguments similar to those used to prove that $\tau(\overline{G})=\left< x,y,v,w\right>$, we show that $\tau(\overline{BS})=\left<z,r\right>$. Since $\gcd(a',b')=1$, there exists $\alpha$, $\beta\in \mathbb{Z}$ such that $\alpha a'+\beta b'=1$ Consequently, $\varphi_{\tau}(x^{\beta}y^{\alpha})=z$ and $\varphi_{\tau}(v^{\beta}w^{\alpha})=r$, which implies that $\varphi_{\tau}$ is an epimorphism.

Now we show that $\ker(\varphi)$ is characteristic in $\overline{G}$. Since elements of both groups $\overline{G}$ and $\overline{BS}$ have unique normal forms, it follows that $|\tau(\overline{G})|=a^2|c-d|^2$ and 
$|\tau(\overline{BS})|=|c'-d'|^2$. By using that $|c-d|=m|c'-d'|$, Lagrange's Theorem and the isomorphism $\tau(\overline{BS})\simeq \tau(\overline{G})/\ker(\varphi)$, we obtain
\[
a^2m^2|c'-d'|^2=a^2|c-d|^2=|\tau(\overline{G})|=|\tau(\overline{BS})||\ker(\varphi)|=|c'-d'|^2|\ker(\varphi)|,
\] so $|\ker(\varphi)|=a^2m^2$. Since $m$ is a Hall divisor of $|c-d|$, $m^2$ is clearly a Hall divisor of $|c-d|^2$. By Lemma \ref{lemanumerodesylow}, since $P(a)\cap P(|c-d|) \subset P(m)$ by hypothesis, $a^2m^2=|\ker(\varphi)|$ is a Hall divisor of $a^2|c-d|^2=|\tau(\overline{G})|$. By basic Sylow theory, $\ker(\varphi)$ must be a product of Sylow subgroups of $\tau(\overline{G})$ (the ones related to the prime numbers involved in $a^2m^2$). Since these subgroups are all characteristic, $\ker(\varphi)$ is characteristic in $\tau(\overline{G})$, so it must also be characteristic in $\overline{G}$, since $\tau(\overline{G})$ is characteristic in $\overline{G}$.

Finally, in both cases of the statement, Theorem \ref{DaciDekimpe} ensures that the group $\overline{BS}$ has $R_\infty$. Thus, by Corollary \ref{characteristic}, $\overline{G}$ has $R_\infty$, which implies the $R_{\infty}$-nilpotency degree of $G$ is $2$, as desired.
\end{proof}

\section{2-generated lollipop GBS groups}\label{2generated}

In this section, we compute the $R_\infty$-nilpotency degree of all 2-generated lollipop GBS groups, by reducing the problem to the case of Baumslag--Solitar groups and applying Theorem \ref{DaciDekimpe} and its refinement, Theorem \ref{refinement}.

Levitt \cite[Theorem 3.2]{Levitt1} established that the rank $\textnormal{rk}(G)$ of an arbitrary GBS group $G$ is given by $\beta(\Gamma)+\mu(\Gamma)$, where $\beta(\Gamma)$ denotes the topological rank of the associated graph $\Gamma$ and $\mu(\Gamma)$ is the \emph{plateaunic number} of $\Gamma$.
For $G=\textnormal{loll}(a,b,c,d)$, with $a>1$, we have $\textnormal{rk}(G)=2$ (i.e. $G$ is $2$-generated) if and only if $(i)$ $\gcd(b,c,d)=1$ and $(ii)$ if $p$ is prime and $p|b$, then $p|c$ or $p|d$. 

A relevant question in this context is whether the assumption of $G=\textnormal{loll}(a,b,c,d)$ being 2-generated is too restrictive. For instance, is every 2-generated lollipop GBS group necessarily isomorphic to a Baumslag--Solitar group? We answer this question in the negative, by constructing an infinite family of counterexamples: consider four pairwise distinct prime numbers $p,q,r,s$ and let $G=\textnormal{loll}(a,b,c,d)=\textnormal{loll}(a,pq,pr,qs)$, with $a>1$. $G$ is clearly 2-generated. Moreover, none of the numbers $b$, $c$ and $d$ divides any of the others. Therefore, $G$ is \emph{strongly slide-free} (see \cite{Forester}), implying that the lollipop graph of $\mathbb{Z}$'s is the unique reduced labeled graph representing $G$. If $G$ were isomorphic to a Baumslag--Solitar group, a labeled loop representing $G$ would also exist, yielding a contradiction. Hence, the family of 2-generated lollipop GBS groups is strictly larger than the family of Baumslag--Solitar groups, so studying its $R_\infty$-nilpotency degree is relevant.

Although not every 2-generated lollipop GBS group is a Baumslag--Solitar group, we will show that their nilpotent quotients are isomorphic to those of certain Baumslag--Solitar groups. To establish this, we require the following theorem, which follows from Theorem 2.1 of \cite{MagKarSol}, since $\gamma_n(G)$ is normally generated by all left-normed $n$-fold commutators.

\begin{theorem}\label{quociente}
    Let $G$ be a group with presentation
    $G=\langle x_1,\dots ,x_r \: | \: R \rangle$. Then, for an arbitrary $n\geq 2$,
    \[ \dfrac{G}{\gamma_n(G)}=\langle x_1,\dots ,x_r \: | \: R, \: [w_1,\dots ,w_n] \: \mbox{for all } w_i\in F(x_1,\dots, x_r) \rangle, \]
    where $F(x_1,\dots ,x_r)$ is a free group on $x_1,\dots ,x_r$.
\end{theorem}

\begin{lemma} \label{potenciascomutador}
    Let $G=\textnormal{loll}(a,b,c,d)$ be a lollipop GBS group. Then, for every $n\in \mathbb{Z}$, we have $[t,x^{nac}]=[t,x^{ac}]^n$.

    \begin{proof}
    We use the following group identity: $[x,yz]=[x,y][x,z][z,x,y]$, which holds for any elements $x,y,z$ of any group $G$. For $n=0,1$, the result is clear. Suppose that $n\geq 2$ and that the lemma holds for $n-1$. Using the defining relations of $G$, we obtain
    \begin{align*}
    [t,x^{nac}] & =[t,x^{(n-1)ac}x^{ac}] \\
    & =[t,x^{(n-1)ac}][t,x^{ac}][x^{ac},t,x^{(n-1)ac}] \\
    & =[t,x^{ac}]^{n-1}[t,x^{ac}][x^{-a(d-c)},x^{(n-1)ac}] \\
    & =[t,x^{ac}]^n.
    \end{align*} Thus, the lemma is true for any $n\geq 0$. For $n=-1$, the lemma also holds:
    \[[t,x^{-ac}]=x^{-ac}[x^{ac},t]x^{ac}=x^{-ac}x^{-a(d-c)}x^{ac}=x^{-a(d-c)}=[t,x^{ac}]^{-1}.\]
Suppose that $n\leq -2$ and that the lemma holds for $n+1$. Following a similar computation as in the previous case, we obtain
    \begin{align*}
    [t,x^{nac}] & =[t,x^{(n+1)ac}x^{-ac}] \\
    & =[t,x^{(n+1)ac}][t,x^{-ac}][x^{-ac},t,x^{(n+1)ac}] \\
    & =[t,x^{ac}]^{n+1}[t,x^{ac}]^{-1}[x^{a(d-c)},x^{(n+1)ac}] \\
    & =[t,x^{ac}]^n.
    \end{align*} Thus, the lemma holds for any $n \in \mathbb{Z}$.
    \end{proof}
\end{lemma}

\begin{lemma}\label{auxiliar}
    Let $G=\textnormal{loll}(a,b,c,d)$ be a lollipop GBS group. Then, for every $k\geq 0$, $x^{a(d-c)^k}\in \gamma_{k+1}(G)$.

    \begin{proof}
        We use induction on $k$. The case $k=0$ is clear and case $k=1$ follows from the relations on $G$:
        \[
        x^{a(d-c)}=y^{b(d-c)}=y^{bd-bc}=ty^{bc}t^{-1}y^{-bc}=tx^{ac}t^{-1}x^{-ac}=[t,x^{ac}] \in \gamma_2(G).
        \]
        Suppose $k\geq 2$ and the result is valid for $k-1$, that is, $x^{a(d-c)^{k-1}}\in \gamma_k(G)$. Thus, by Lemma \ref{potenciascomutador},
        \[ x^{a(d-c)^k}=(x^{a(d-c)})^{(d-c)^{k-1}}=[t,x^{ac}]^{(d-c)^{k-1}}=[t,x^{a(d-c)^{k-1}c}]\in \gamma_{k+1}(G).\]
    \end{proof}
\end{lemma}

\begin{theorem}\label{isomorfismodosquocientes}
Let $G=\textnormal{loll}(a,b,c,d)$ be a $2$-generated lollipop GBS group. Thus, for every $n\geq 2$,
    \[ \dfrac{G}{\gamma_n(G)}\simeq \dfrac{BS(ac,ad)}{\gamma_n(BS(ac,ad))}.\]

\begin{proof}
We will prove that both nilpotent quotients mentioned above share a common presentation and are therefore isomorphic. Recall that we have the following presentation:
    \[ BS(ac,ad)=\langle x,t \: | \: tx^{ac}t^{-1}=x^{ad}\rangle.\]
    Therefore, by Theorem \ref{quociente}, 
    \begin{equation} \label{presbsquoc}
        \dfrac{BS(ac,ad)}{\gamma_n(BS(ac,ad))}=\langle x,t \: | \: tx^{ac}t^{-1}=x^{ad}, [w_1,\dots ,w_n],\: w_i \in F(x,t)  \rangle.
    \end{equation}
For the group $G$, we have
    \[ \dfrac{G}{\gamma_n(G)}=\langle x,y,t \: | \: x^a=y^b, \: ty^ct^{-1}=y^d, [z_1,\dots ,z_n],\: z_i \in F(x,y,t)\rangle,\] where $F(x,t)$ and $F(x,y,t)$ denote the associated free groups.

Note that a similar argument as in Lemma \ref{auxiliar}, applied to the relation $ty^ct^{-1} = y^d$,
gives $y^{(d-c)^k} \in \gamma_{k+1}(G)$ for all $k \geq 0$. Since $G$ is $2$-generated, our earlier analysis shows that $\gcd(b,c,d)=1$ and that, for any prime $p$, if $p|b$ then $p|c$ or $p|d$. From this it follows that $\gcd(b,d-c)=1$ and, therefore, $\gcd(b,(d-c)^{n-1})=1$. Thus, there are $\alpha, \theta\in \mathbb{Z}$ such that
\begin{equation}\label{equation2gen}
    \alpha b+\theta (d-c)^{n-1}=1.
\end{equation} Hence,
\[
  y=y^{\alpha b}y^{\theta (d-c)^{n-1}}=x^{\alpha a}y^{\theta (d-c)^{n-1}}=x^{\alpha a} \pmod{\gamma_n(G)}.
\]
This way, by performing the Tietze transformation \cite[p.~89]{LyndonSchupp} of replacing $y$ by $x^{a\alpha}$, the relations in the presentation of $G/\gamma_n(G)$ become
\[
  x^a=y^b \Longleftrightarrow x^a=x^{\alpha ab} \Longleftrightarrow x^{a(1-\alpha b)}=1 \Longleftrightarrow x^{a\theta(d-c)^{n-1}}=1,
\]
\[
 ty^ct^{-1}=y^d \Longleftrightarrow y^{d-c}=[t,y^c] \Longleftrightarrow x^{\alpha a(d-c)}=[t,x^{\alpha ac}],
\]
and the relations $[z_1,\dots ,z_n]$ become all relations $[z_1,\dots ,z_n]$, with $z_i$ words in letters $x,t$. Therefore
    \[ \dfrac{G}{\gamma_n(G)}=\langle x,t \: | \: x^{a\theta(d-c)^{n-1}}=1, \: x^{\alpha a(d-c)}=[t,x^{\alpha ac}], \: [z_1,\dots ,z_n], \: z_i\in F(x,t)\rangle. \]
    Now, by Equation \ref{equation2gen} we clearly have
\[
a(d-c)=a(d-c)(\alpha b+\theta (d-c)^{n-1})= \alpha a(d-c) b+a\theta (d-c)^{n}.
\]
Therefore, using Lemma \ref{potenciascomutador}, Equation \ref{equation2gen} and the relations on $G/\gamma_n(G)$, it follows that
    \begin{align*}
        x^{a(d-c)} & =x^{\alpha a(d-c) b}x^{a\theta (d-c)^{n}} \\
                & =x^{\alpha a(d-c) b} \\
                   & =[t,x^{\alpha ac}]^b \\
                   & =[t,x^{ac}]^{\alpha b} \\
                   & =[t,x^{ac}]^{1-\theta (d-c)^{n-1}} \\
                   & =[t,x^{ac}][t,x^{ac}]^{-\theta(d-c)^{n-1}} \\
                   & =[t,x^{ac}][t,x^{ac(d-c)^{n-2}}]^{-\theta(d-c)}\\
                   & =[t,x^{ac}],
    \end{align*}
    since $x^{a(d-c)^{n-2}}=y^{b(d-c)^{n-2}}\in \gamma_{n-1}(G)$. 
    So, we have the relation $tx^{ac}t^{-1}=x^{ad}$ in $G/\gamma_n(G)$. Now we will prove that the group given by the presentation
    \[ \langle x,t \: | \: tx^{ac}t^{-1}=x^{ad}, \: [z_1,\dots ,z_n], \: z_i\in F(x,t)\rangle\]
   is exactly $G/\gamma_n(G)$. In fact, since $tx^{ac}t^{-1}=x^{ad}$, following the same arguments in the proof of Lemma \ref{potenciascomutador}, we have that $[t,x^{nac}]=[t,x^{ac}]^n$. Thus,
    \[ x^{\alpha a(d-c)}=(x^{a(d-c)})^{\alpha}=[t,x^{ac}]^{\alpha}=[t,x^{\alpha ac}].\]
    Denoting $y=x^{\alpha a}$, we have $y^{d-c}=[t,y^c]$. For the relation $x^a=y^b$, it is sufficient to prove that $x^{a\theta (d-c)^{n-1}}$=1, and this is true by Lemma \ref{auxiliar}. Therefore, using Tietze transformations \cite[p.~89]{LyndonSchupp}, we can rewrite the presentation of $G/\gamma_n(G)$ as:
    \[ \dfrac{G}{\gamma_n(G)}=\langle x,t \: | \: tx^{ac}t^{-1}=x^{ad}, \: [z_1,\dots ,z_n], \: z_i\in F(x,t)\rangle, \]
   which is precisely the presentation in Equation \ref{presbsquoc}. Therefore, we have the desired isomorphism.
    \end{proof}
\end{theorem}

\begin{theorem}\label{2geradosfinal}
   Let $G=\textnormal{loll}(a,b,c,d)$ be a $2$-generated lollipop GBS group. Then $G$ and $BS(ac,ad)$ have the same $R_{\infty}$-nilpotency degree $r$. Specifically, let $p$ denote the largest integer such that $2^p$ divides $2\frac{c}{\gcd(c,d)}+2$. We have:

\begin{itemize}
    \item If $d<0$ and $d \neq -c$, then $r=2$.
    \item If $d=c$ or $d=-c$, then $r=\infty$.
    \item If $d-c=\gcd(c,d)$, then $r=\infty$.
    \item If $d-c=2\gcd(c,d)$ and $ac,ad$ are not both even, then $r=2$.
    \item If $d-c=2\gcd(c,d)$ and $ac,ad$ are both even, then $2\leq r\leq p+2$.
    \item If $d-c\geq 3\gcd(c,d)$, then $r=2$.
\end{itemize}

\begin{proof}
By Theorem \ref{isomorfismodosquocientes}, $G$ and $BS(ac,ad)$ must have the same $R_{\infty}$-nilpotency degree $r$.
In the first case, we have $ad<0$ (since $a>0$) and $ad\neq -ac$. Thus, by Theorem \ref{DaciDekimpe}, the $R_{\infty}$-nilpotency degree of $BS(ac,ad)$ is $2$. In the second case, we have either $ad=ac$ or $ad=-ac$. Thus, by Theorem \ref{DaciDekimpe}, the $R_{\infty}$-nilpotency degree of $BS(ac,ad)$ is $r=\infty$. For the other cases, it is enough to note that, if $d-c=k\gcd(c,d)$, then $ad-ac=ka\gcd(c,d)=k\gcd(ac,ad)$. Thus, the $R_{\infty}$-nilpotency degree of $G$ (and of $BS(ac,ad)$) is given either by Theorem \ref{DaciDekimpe} or by Theorem \ref{refinement} (in the fourth case).
\end{proof}
\end{theorem}

\begin{remark}
Theorem \ref{pirulitosindice2} is subsumed by Theorem \ref{Sylow1}, as one can easily check. 
However, none of our three main theorems (\ref{notdivideb}, \ref{Sylow1} and \ref{2geradosfinal}) are subsumed by the other two. In fact, $\textnormal{loll}(7,2,54,174)$ satisfies Theorem \ref{notdivideb} but neither Theorem \ref{Sylow1} nor Theorem \ref{2geradosfinal}; the group $\textnormal{loll}(4,15,8,-16)$ satisfies Theorem \ref{Sylow1} but none of the other two; finally, the group $\textnormal{loll}(17,15,273,715)$ satisfies Theorem \ref{2geradosfinal} but none of the other two. We finish this remark by noting that $\textnormal{loll}(17,30,273,715)$, for instance, does not satisfy any of our results and we do not know its $R_\infty$-nilpotency degree.
\end{remark}

\begin{remark}\label{howmuch}
After computing the $R_\infty$-nilpotency degree for $\textnormal{loll}(a,b,c,d)$ in a variety of different conditions, it was not clear to us ``how much'' we had solved, that is, what was the percentage of lollipop GBS groups with determined $R_\infty$-nilpotency degree. So, we used the GAP \cite{GAP4} and R \cite{R} softwares to have an idea in finite sample spaces. Given $n\geq 2$, we implemented the ``$n$-algorithm'' as follows: ``\emph{given integers $a,b,c,d$ with $1 \leq |a|,|b|,|c|,|d| \leq n$ (and $1< |a|,|b|$, so that $\textnormal{loll}(a,b,c,d)$ is not automatically a $BS$ group), return `true' if $\textnormal{loll}(a,b,c,d)$ satisfies all hypotheses of at least one of theorems \ref{notdivideb}, \ref{Sylow1} and \ref{2geradosfinal}, and return `false' otherwise}''. The next graph and table show our results for different values of $n$ with $2 \leq n \leq 50$.

\begin{figure}[htb!]
\centering
\includegraphics[scale=0.55]{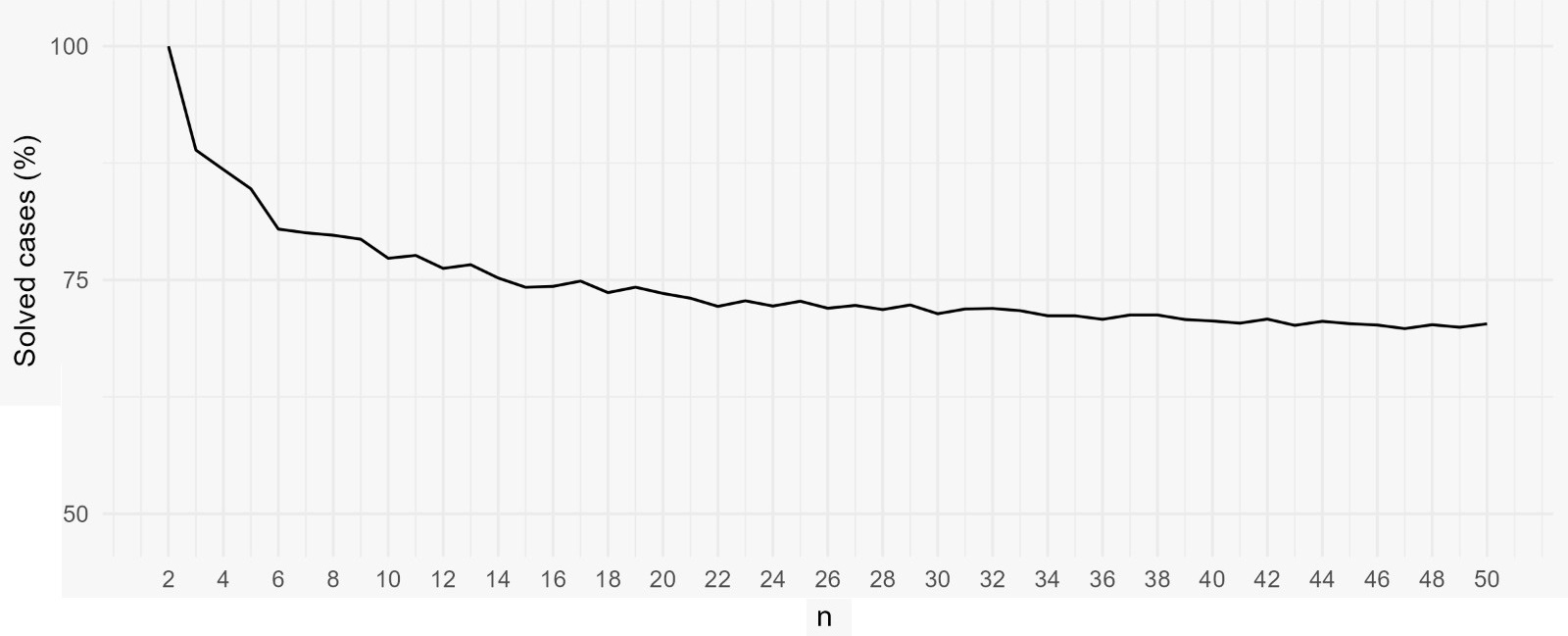}
\caption{percentage of solved cases for the $n$-algorithm}
\end{figure}

\vspace{0.5cm}

\begin{center}
\begin{tabular}{|c|c|c|c|}
\hline
n  & Number of groups ($(2n)^2(2n-2)^2$)  & Number of solved cases & Percentage \\
  \hline 
10  & 129,600 & 100,208 & 77.32\%\\
\hline
20  & 2,310,400 & 1,699,616 & 73.56\% \\
\hline
30    & 12,110,400 & 8,645,728 & 71.39\% \\
\hline
40    & 38,937,600 & 27,405,072 & 70.38\% \\
\hline
50    & 96,040,000 & 67,170,176 & 69.93\% \\
\hline
\end{tabular}
\end{center}

\vspace{0.5cm}

These results indicate a substantial coverage of cases, at least for small values of $n$. We informally believe that a comparable proportion of cases may remain covered for larger values of $n$.

\end{remark}

\end{document}